\documentclass[ijoc,sglanonrev,letterpaper,onecolumn]{informs4}

\OneAndAHalfSpacedXII
\usepackage[letterpaper,margin=1in]{geometry}
\usepackage{natbib}
\bibpunct[, ]{(}{)}{,}{a}{}{,}%
\def\bibfont{\fontsize{9}{11.3}\selectfont}%
\usepackage{fix-cm}
\usepackage{anyfontsize}

\usepackage{makecell, float, array, booktabs, amsmath, amssymb, mathtools, graphicx}
\usepackage{xcolor}

\usepackage[linesnumbered,ruled,vlined]{algorithm2e}
\usepackage[T1]{fontenc}
\usepackage{longtable}
\usepackage{hyperref}
\hypersetup{
    colorlinks,
    linkcolor={red!50!red},
    citecolor={blue!50!blue},
    urlcolor={blue!80!blue},
    pdfinfo={
        Title={},
        Author={},
        Subject={},
        Keywords={},
        Creator={},
        Producer={}
    }
}
\usepackage{enumitem}
\setlist[itemize]{label=\textbullet}
\graphicspath{{./figures/}}
\renewenvironment{proof}[1][Proof]{\par\textit{#1.}\enspace\ignorespaces}{\par}

\TheoremsNumberedThrough
\EquationsNumberedThrough

\begin{document}

\RUNAUTHOR{You et al.}
\RUNTITLE{L-PDDOIs for Stabilizing Column Generation}

\TITLE{Learned Pairwise Deep Dual-Optimal Inequalities for Stabilizing Column Generation}

\ARTICLEAUTHORS{%
    \AUTHOR{%
        Zhengzhong Ricky You\textsuperscript{1},
        Bo Tang\textsuperscript{2,*},
        Haoran Liu\textsuperscript{3},
        Baichuan Mo\textsuperscript{1}
    }%
    \AFF{%
        \textsuperscript{1}Department of Civil Engineering, Tsinghua University, Beijing 100084, China, \EMAIL{ricky.you.or@gmail.com, bmo@tsinghua.edu.cn}\\
        \textsuperscript{2}MIT Sloan School of Management, Massachusetts Institute of Technology, 100 Main Street, Cambridge, MA 02142, USA, \EMAIL{botang@mit.edu}\\
        \textsuperscript{3}Department of Industrial and Systems Engineering, University of Florida, Gainesville, FL 32611, USA, \EMAIL{haoran.liu@ufl.edu}\\
        \textsuperscript{*}Corresponding author.
    }%
}

\ABSTRACT{
Column generation (CG) is central to many large-scale optimization algorithms, including branch-price-and-cut methods for vehicle routing problems, but unstable dual solutions can substantially slow its convergence. Existing deep dual-optimal inequalities can reduce this instability by restricting the dual space. Their construction, however, typically relies on problem-specific exchange arguments that are difficult to establish for routing problems with capacity limits, time windows, and other resource constraints.
We introduce learned pairwise deep dual-optimal inequalities (L-PDDOIs), a learning framework that predicts pairwise orderings between dual variables and incorporates their primal counterparts directly into the master problem.
To construct training labels, the framework samples optimal dual solutions and selects pairwise order relations that hold simultaneously on a sufficiently large common subset of the samples. A classifier then assigns a score to each candidate relation. Because conflicts and redundancies among the predicted relations can impair performance, graph-based postprocessing filters and compresses the candidate set before deployment. We further introduce a recovery procedure that selectively relaxes learned inequalities and provides a certificate when the baseline CG bound has been restored.
On the main test sets for the capacitated vehicle routing problem and the vehicle routing problem with time windows, direct deployment of L-PDDOIs reduces the geometric mean root CG time by \(89.7\%\) and \(93.9\%\), respectively, while incurring mean bound losses of only \(1.3\%\) and \(0.5\%\). The recovery procedure retains corresponding time reductions of \(54.8\%\) and \(83.1\%\), respectively, while guaranteeing no loss in the CG bound.
}


\maketitle

\setlength{\abovedisplayskip}{6pt plus 2pt minus 2pt}
\setlength{\belowdisplayskip}{6pt plus 2pt minus 2pt}
\setlength{\abovedisplayshortskip}{4pt plus 1pt minus 1pt}
\setlength{\belowdisplayshortskip}{4pt plus 1pt minus 1pt}
\setlength{\jot}{1.5pt}

\section{Introduction}\label{sec_intro}

Exact solution methods for vehicle routing problems (VRPs) often rely on branch-price-and-cut (BPC) frameworks, which have been applied across a broad range of VRP classes \citep{costa2019exact}. Their effectiveness in routing derives largely from the strong linear programming (LP) bounds provided by the underlying set-partitioning formulation. Because this formulation contains an exponential number of route variables, its LP relaxation is typically solved by column generation (CG), which alternates between solving a restricted master problem (RMP) and pricing new routes until no route with negative reduced cost remains. CG often dominates the computation in BPC because each pricing call typically requires solving an elementary shortest path problem with resource constraints, which is NP-hard \citep{feillet2010tutorial,spliet2023technical}. Even when elementarity is relaxed, as in the \(ng\)-route relaxation, the resulting pricing problem remains computationally demanding \citep{baldacci2011new,costa2019exact}. Improving exact routing algorithms therefore often depends on accelerating CG convergence. A central obstacle is dual instability. Successive RMP solutions may exhibit large oscillations in their dual multipliers, leading to tailing off and slow convergence \citep{lubbecke2005selected,rousseau2007interior,gschwind2016dual}. Similar difficulties have been documented in cutting stock \citep{ben2006dual}, bin packing \citep{gschwind2016dual}, and airline crew scheduling \citep{borndorfer2006column}.

Routing set-partitioning RMPs are often highly primal-degenerate because a basic solution contains as many basic variables as there are linearly independent partitioning constraints, whereas only a small number of route variables are typically positive. Consequently, many basic route variables may take value zero, and multiple degenerate bases may represent the same primal solution. Because simplex dual multipliers depend on the selected basis, different degenerate bases associated with the same or nearly identical primal solution can yield substantially different dual multiplier vectors. Moreover, when only a small number of route constraints are active at dual optimality, the optimal dual face may retain considerable freedom. Small changes in the RMP or its basis can therefore induce large variations in the dual multipliers, resulting in dual oscillation and unstable pricing behavior in CG.

Dual-optimal inequalities (DOIs) and deep dual-optimal inequalities (DDOIs) impose carefully chosen inequalities on the dual space that preserve the relevant dual optimum; by contracting the dual search region, they can accelerate CG without degrading the target dual bound \citep{ben2006dual,gschwind2016dual,haghani2022smooth,guo2025nested}. However, deriving DDOIs that are both sufficiently strong and provably compatible with an optimal dual solution remains difficult in routing. Existing constructions typically rely on exchange arguments specific to the problem, including pattern replacement arguments in cutting stock and packing models, overcoverage and undercoverage trades used by smooth and flexible DOIs, and inexpensive item swaps or detour operations along active routes \citep{ben2006dual,gschwind2016dual,haghani2022smooth,yarkony2021detour}. Such arguments require a valid and cheaply computable compensation rule, i.e., after exchanging an item or relaxing coverage, one must still obtain a feasible column or an explicitly bounded surrogate cost. The compensation requirement is difficult to satisfy in general routing, where a seemingly local customer exchange can violate global route feasibility through capacity, time windows, precedence, or synchronization requirements. Certifying the compensating penalty may itself require solving another pricing problem with resource constraints; otherwise, the compensating bound must be conservative, preserving validity but potentially making the resulting inequality ineffective. Such structural barriers leave limited room for broadly deployable DDOIs in routing and motivate methods that identify candidate dual inequalities without manually derived exchange rules for each VRP variant.

We address this challenge with learned pairwise deep dual-optimal inequalities (L-PDDOIs), a learning framework for constructing simple pairwise dual inequalities of the form \(p_i \le p_j\), where \(p_i\) and \(p_j\) are set-partitioning dual variables associated with customers. The learning component formulates pair ranking as a supervised binary classification task over ordered customer pairs \((i,j)\). For each pair, the model uses routing features \(\mathbf f_{ij}\) to produce a score \(s_{ij}\in[0,1]\), where a larger score indicates stronger model support for including \(p_i\le p_j\) in the candidate set. The training labels identify pairwise inequalities that hold consistently across sampled optimal dual solutions.

The learner scores each ordered pair independently, which can produce structural inconsistencies. In particular, independent scoring does not guarantee transitivity. Although the predicted inequalities \(p_i \le p_j\) and \(p_j \le p_k\) jointly imply \(p_i \le p_k\), the model may reject the latter inequality. Moreover, the simultaneous predictions \(p_i \le p_j\), \(p_j \le p_k\), and \(p_k \le p_i\) force \(p_i = p_j = p_k\), forming what we later define as a cycle, even though equalities are excluded when constructing the labels. To address these issues, graph repair removes predictions as needed to obtain an acyclic and inference-closed subgraph. L-PDDOIs therefore combine learned pairwise predictions with deployment decisions that account for the structure of the optimization model. Although graph repair controls the joint consistency of the predicted inequality system, it does not guarantee preservation of the baseline CG bound. The distinction gives rise to two deployment regimes. Direct deployment prioritizes computational speed at the possible expense of bound strength. After exact pricing, the restricted dual still yields a valid lower bound, although it may be weaker than the baseline. Recovery is introduced when the original bound, or the corresponding solution of the original master problem, must be restored before branching. It releases candidate inequalities until none remain active, at which point full bound recovery is achieved.

\paragraph{Main Contributions.}
This paper makes the following three contributions to learning and deploying structural dual inequalities in CG.
\begin{itemize}
    \item \textit{We introduce L-PDDOIs as a learnable family of structural dual inequalities for CG.} L-PDDOIs predict pairwise order relations \(p_i\le p_j\) between customer dual variables and incorporate their primal counterparts into the RMP through artificial columns. We define deep dual optimality jointly for the full inequality system and develop consistency-aware probabilistic set construction (CAPSC), which samples the optimal dual face and constructs jointly compatible labels for training a supervised pair classifier.

    \item \textit{We develop repair, compression, and recovery mechanisms for deploying learned inequalities with structural and bound guarantees.} We formulate an integer program that retains a large subset of predicted relations while enforcing acyclicity and inference closure, and then apply transitive reduction to remove redundant inequalities without changing the induced dual feasible region. When the original CG bound must be recovered, inequalities associated with active artificial variables are iteratively released using progressively stronger pricing stages.

    \item \textit{We provide extensive computational evidence on the capacitated vehicle routing problem (CVRP) and vehicle routing problem with time windows (VRPTW).} Direct L-PDDOI deployment reduces the geometric mean root CG time by \(89.7\%\) for the CVRP and \(93.9\%\) for the VRPTW, with mean bound losses of only \(1.3\%\) and \(0.5\%\), respectively. Recovery restores the baseline bounds while retaining time reductions of \(54.8\%\) and \(83.1\%\), substantially exceeding those achieved by automatic dual price smoothing. Graph repair reduces the mean bound loss from \(9.5\%\) to \(1.3\%\) on the CVRP and from \(15.0\%\) to \(0.5\%\) on the VRPTW, while transitive reduction preserves these bounds and substantially reduces the number of deployed inequalities. Transfer experiments on CVRPLIB and Gehring--Homberger benchmarks show that the runtime gains persist on unseen instance families and larger problems.
\end{itemize}

The remainder of the paper is organized as follows. Section~\ref{sec_review} reviews stabilization methods for routing problems, manually derived DDOIs, and learning approaches to dual stabilization. Section~\ref{sec_setting} introduces the set-partitioning master problem, its dual formulation, and the interpretation of CG in dual space. Section~\ref{sec_pair_layer} presents the L-PDDOI framework, including label construction, pair prediction, graph repair, transitive reduction, and recovery.
Section~\ref{sec_exp_design} evaluates the performance of L-PDDOIs through case studies on the CVRP and VRPTW.
Section~\ref{sec_conclusion} concludes the paper. The \hyperref[sec:ec:start]{electronic companion (EC)} provides the proofs, feature construction and instance generation details, and additional computational results.

\section{Related Literature on Stabilization}\label{sec_review}

We review three streams of literature most closely related to our work. We first summarize the development of dual stabilization methods for CG. We then discuss DOIs and DDOIs, which provide the structural foundation for L-PDDOIs. Finally, we review machine learning approaches to dual stabilization and clarify how our work complements these approaches.

\subsection{Dual Stabilization Methods}\label{subsec_review_routing}

Dual stabilization is a classical response to the slow convergence of CG caused by degeneracy and oscillatory dual solutions. Early methods address this issue by explicitly restricting dual movement. The boxstep method confines the dual vector to a neighborhood of a stability center, thereby preventing consecutive master iterations from moving too far from a trusted reference point \citep{marsten1975boxstep}. Although this box constraint directly suppresses dual oscillation, it may be overly restrictive. Later stabilized CG methods therefore replace fixed dual bounds with softer regularization mechanisms. Penalty, proximal, and trust-region terms guide the dual search toward a stability center while allowing the stabilization region to evolve as the algorithm progresses \citep{ben2009choice,briant2008comparison}. These softer approaches offer greater flexibility but depend on the choice of penalty parameters, update rules, and the quality of the reference point.

Another approach selects more informative dual solutions for pricing while controlling their movement. Weighted smoothing and in-out separation combine current and reference vectors to avoid unstable prices \citep{wentges1997weighted}. Building on dual price smoothing, \citet{pessoa2018automation} connect it with in-out separation and combine smoothing with penalty stabilization through self-adjusting parameters. Their method reduces the need to tune parameters for each instance and makes stabilization easier to use in BPC implementations. In parallel, interior point and analytic center approaches favor well-centered dual solutions over extreme dual optima and have been studied in CG, BPC, and vehicle routing settings \citep{rousseau2007interior,munari2013using,gondzio2016large,karimi2015analytic}. A distinct structural approach adds valid inequalities that restrict the dual space. Extra dual cuts and DOIs reduce the dual search region without changing the CG bound \citep{valerio2005using,ben2006dual}, while deeper or smoother variants impose stronger inequalities by preserving at least one optimal dual solution or by allowing controlled coverage exchanges \citep{gschwind2016dual,haghani2022smooth}. Because this structural perspective is most closely related to L-PDDOIs, the next subsection reviews DOI and DDOI constructions in greater detail.

\subsection{Dual-Optimal and Deep Dual-Optimal Inequalities}\label{subsec_dDOIs_review}

\citet{ben2006dual} distinguish DOIs, which preserve all optimal dual solutions, from DDOIs, which need only preserve at least one. The latter permit stronger contraction of the dual region but require a problem-specific argument establishing preservation of an optimal dual point. \citet{gschwind2016dual} extend the idea to bin packing, cutting stock, vector packing, vertex coloring, and bin packing with conflicts, and further discuss dynamic separation and recovery mechanisms for using strong inequalities when their validity or usefulness must be controlled adaptively. Subsequent work develops compensation rules for different master structures. Flexible DOIs give rebates for overcoverage in set packing models \citep{lokhande2020accelerating}; detour DOIs derive dual bounds from swap or detour arguments in routing and location models \citep{yarkony2021detour}; and smooth DOIs allow controlled undercoverage in exchange for over-inclusion in set covering models, including the CVRP \citep{haghani2022smooth}. More recent work studies nested set covering and set packing structures and shows that stabilization based on DOIs and DDOIs should also be analyzed together with primal degeneracy \citep{guo2025nested}.

\subsection{Machine Learning Approaches to Dual Stabilization}\label{subsec_ml_review}

A growing literature uses machine learning to provide dual information for CG, although the predicted objects and their algorithmic roles vary across methods. \citet{kraul2023machine} use predicted optimal dual values to configure a stabilized CG scheme, \citet{sugishita2024use} generate initial dual values to warm-start CG, and \citet{shen2024adaptive} combine predictions of optimal dual solutions with adaptive stabilization. In contrast, L-PDDOIs learn pairwise order inequalities, impose the selected relations directly in the master, and use the activity of the corresponding artificial variables after exact pricing as a release signal during recovery. The proposed construction does not require each relation to hold over the entire optimal dual face. Instead, CAPSC selects a system of relations jointly supported by a retained subset of sampled optimal points, while joint deep dual optimality requires the deployed system to preserve at least one optimal dual solution.

A separate stream uses machine learning to reduce the cost of generating, prioritizing, or admitting columns. These methods select generated columns \citep{morabit2021machine}, select arcs in pricing networks \citep{morabit2023machine}, rank pricing subproblems \citep{koutecka2025machine}, or learn policies for multiple-column selection, pricing, and heuristic selection \citep{yuan2024reinforcement,abouelrous2025reinforcement,xu2025enhancing}. They can affect the dual sequence indirectly by changing the columns seen by the RMP, but do not directly shrink the dual region to mitigate oscillation. L-PDDOIs target this complementary source of instability by learning explicit inequalities that constrain the dual trajectory determining reduced costs and driving pricing.

\section{Preliminaries}\label{sec_setting}

Before presenting the proposed method, we formulate the master problem, derive its dual, and explain why CG may oscillate in dual space. We consider the LP relaxation of the standard set-partitioning formulation used in exact vehicle routing methods based on CG \citep{costa2019exact,spliet2023technical}. The primal master problem and its dual are
\begin{center}
\small
\refstepcounter{equation}\label{formulation_mp}\edef\masterEq{\theequation}
\refstepcounter{equation}\label{formulation:MP_dual}\edef\masterDualEq{\theequation}
\[
\begin{array}{@{}c@{\quad}c@{\quad}c@{\quad}c@{\quad}c@{}}
\begin{alignedat}{2}
F:\quad \min\quad
& \sum_{r\in\Omega} c_r x_r\\
\text{s.t.}\quad
& \sum_{r\in\Omega} a_{ir}x_r = 1,
&& \quad \forall i\in[m],\\
& x_r \ge 0,
&& \quad \forall r\in\Omega,
\end{alignedat}
&
(\masterEq)
&
\mathrel{\vcenter{\hbox{$\displaystyle\xrightleftharpoons[\hspace{1.25em}\text{\footnotesize primal}\hspace{1.25em}]{\hspace{1.25em}\text{\footnotesize dual}\hspace{1.25em}}$}}}
&
\begin{alignedat}{2}
DF:\quad \max\quad
& \sum_{i\in[m]} p_i\\
\text{s.t.}\quad
& \sum_{i\in[m]} a_{ir}p_i \le c_r,
&& \quad \forall r\in\Omega,
\end{alignedat}
&
(\masterDualEq)
\end{array}
\]
\end{center}

Here, \([m]\) indexes the customer coverage constraints, \(\Omega\) denotes the set of feasible routes, \(c_r\) is the cost of route \(r\), \(a_{ir}\in\{0,1\}\) indicates whether route \(r\) visits customer \(i\), \(x_r\) is the corresponding master variable, and \(p_i\) is the unrestricted dual price associated with customer \(i\). When elementarity is relaxed, for example under ng-route relaxation \citep{baldacci2011new}, \(a_{ir}\) can be generalized to \(a_{ir}\in\mathbb{Z}_{\ge 0}\), in which case it represents the number of visits to customer \(i\) by route \(r\).

The primal and dual formulations give two complementary views of CG. The primal column view follows the implementation directly. At iteration \(k\), the RMP obtained from \eqref{formulation_mp} by replacing \(\Omega\) with \(\Omega^k\subseteq\Omega\) is solved, yielding a dual vector \(\mathbf p^k=(p_i^k)_{i\in[m]}\). Pricing then searches for a route \(r\in\Omega\) with negative reduced cost \(\bar c_r(\mathbf p)=c_r-\sum_{i\in[m]}a_{ir}p_i\). If such a route exists, it is added to the RMP; otherwise, \(\mathbf p^k\) satisfies all dual constraints in \eqref{formulation:MP_dual}, and the current RMP solution is optimal for the full master problem. The primal view describes the implementation of CG, but it leaves a natural question unanswered: why can the dual vector change substantially when each iteration adds only one or a few columns? Let \(z^\star\) denote the common optimal value of \eqref{formulation_mp} and \eqref{formulation:MP_dual}, which we call the \emph{baseline CG bound}; equality follows from strong LP duality. The bound is obtained by standard CG over the original route columns without introducing artificial columns. Let \(\mathcal D^\star\) be the optimal face of the full dual at objective value \(z^\star\), and define the current RMP dual region as
\[
\mathcal D^k:=\left\{\mathbf p\in\mathbb R^m:\sum_{i\in[m]}a_{ir}p_i\le c_r,\ \forall r\in\Omega^k\right\}.
\]

We introduce \(\mathcal D^k\) to view pricing as separation in dual space. From this perspective, CG optimizes over \(\mathcal D^k\), uses pricing to identify a violated dual constraint, and either terminates or contracts \(\mathcal D^k\) by intersecting it with the halfspace induced by the priced route. Reoptimization over these changing regions does not move the dual solution along a smooth path. After each contraction, the previous dual optimizer becomes infeasible, and a new optimizer can jump to a distant extreme point or face of the relaxed dual polyhedron. Such jumps are especially common in highly degenerate masters, where many dual solutions are optimal or nearly optimal for the current relaxation. Since pricing evaluates reduced costs using the current dual vector, abrupt changes in \(\mathbf p^k\) can substantially reshape the reduced-cost landscape and lead to oscillatory pricing behavior.

For the pair inequalities used below, let \(\mathcal I:=\{(i,j)\in[m]\times[m]:i\ne j\}\), and for \(\mathcal E\subseteq\mathcal I\) let
\[
\mathcal O(\mathcal E):=\{\mathbf p\in\mathbb R^m:p_i-p_j\le 0,\ \forall (i,j)\in\mathcal E\}
\]
denote the order cone induced by \(\mathcal E\).

\begin{definition}[Jointly deep dual-optimal]
A pair set \(\mathcal E\subseteq\mathcal I\) is \emph{jointly deep dual-optimal} if \(\mathcal D^\star\cap\mathcal O(\mathcal E)\ne\varnothing\).
\end{definition}

The property applies to the inequality set as a whole. Although each inequality in \(\mathcal E\) may individually leave \(\mathcal D^\star\) nonempty, imposing them jointly may exclude all baseline optimal dual solutions. The deployment decision must therefore account for the selected inequalities collectively. We refer to each predicted inequality as a \emph{candidate L-PDDOI}, while L-PDDOI without qualification denotes the complete learning and deployment framework. Because candidate inequalities are not guaranteed to preserve \(z^\star\), Section~\ref{subsec_bound_recovery} provides an ex post condition for certifying whether the deployed set preserves the baseline bound.

\section{Learning and Deploying Pairwise Deep Dual-Optimal Inequalities}\label{sec_pair_layer}

We begin with an overview of the complete L-PDDOI pipeline, explaining how pairwise predictions are transformed through graph-based postprocessing into a compact inequality system for deployment and how the baseline CG bound can subsequently be recovered when required. We then introduce CAPSC, which constructs jointly consistent training labels from sampled optimal dual solutions, followed by the feature representation and learning model used to score candidate pairwise inequalities. Next, we present the graph repair and transitive reduction procedures that enforce structural consistency and eliminate redundant inequalities before deployment. Finally, we describe how the resulting inequalities are incorporated into the master problem and develop the corresponding certification and bound-recovery procedure.

\subsection{Overview of the L-PDDOI Deployment Pipeline}\label{subsec_deployment_overview}

The deployment pipeline begins after the prediction model has been trained. As summarized in Figure~\ref{fig:pair_layer_pipeline}, the pipeline has three stages. First, the trained model \(h_\theta\) maps the pair features to the score matrix \(\mathbf S\), whose entry \(s_{ij}\) scores the candidate L-PDDOI \(p_i-p_j\le 0\). Second, \(\mathbf S\) defines a raw directed graph \(G^0=([m],E^0)\), where \(E^0=\left\{(i,j)\in\mathcal I:s_{ij}\ge \tau_{\mathrm{pred}}\right\}\) for the deployment cutoff \(\tau_{\mathrm{pred}}\). Independent thresholding can create a cycle, which forces unintended equalities, or a chain \(i\to j\to k\) whose implied relation \(i\to k\) was not predicted. The repair step discards arcs as needed so that the retained graph \(G^{\mathrm{rep}}=([m],E^{\mathrm{rep}})\) is acyclic and every relation implied by a retained chain was supported by the raw predictions. Third, transitive reduction removes redundant shortcut arcs while preserving the same reachability relation and order cone, yielding \(G^{\mathrm{tr}}=([m],E^{\mathrm{tr}})\). The final deployed candidate set is \(\mathcal E:=E^{\mathrm{tr}}\), and each \((i,j)\in\mathcal E\) is imposed as \(p_i-p_j\le 0\).

\begin{figure}[htbp]
    \centering
    \caption{Overview of the prediction and postprocessing pipeline for L-PDDOIs.}
    \label{fig:pair_layer_pipeline}
    \vspace{2pt}
    \includegraphics[width=\linewidth]{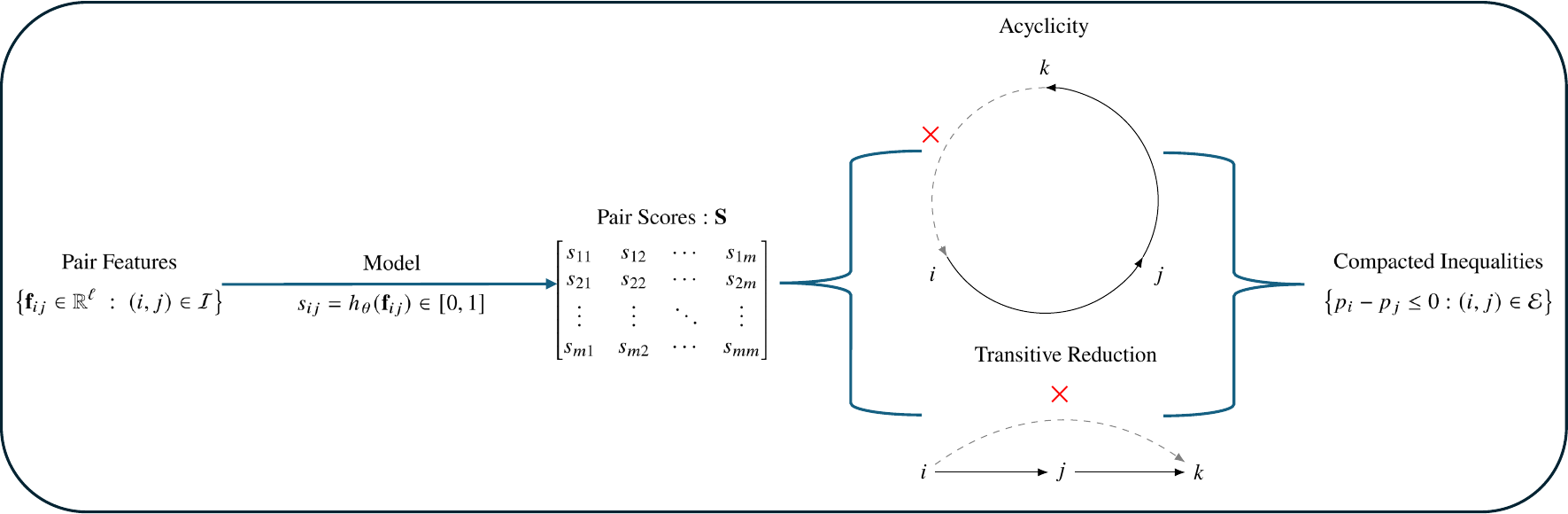}
    \vspace{-5pt}
\end{figure}

Throughout the remainder of this section, we use the following graph notation. For any arc set \(E\subseteq\mathcal I\), let \(G(E)=([m],E)\). For each \(i\in[m]\), its outgoing neighborhood is \(N_E^+(i):=\{j\in[m]\mid(i,j)\in E\}\), and its reachable set is \(R_E^+(i):=\{j\in[m]\setminus\{i\}\mid i\leadsto j\text{ in }G(E)\}\), where \(i\leadsto j\) indicates that \(G(E)\) contains a directed path from \(i\) to \(j\).
The induced reachability relation is \(\mathcal R(E):=\{(i,j)\in[m]\times[m]\mid j\in R_E^+(i)\}\). We say that \(G(E)\) satisfies \emph{inference closure} if every reachable pair is explicitly included, i.e., \(j\in R_E^+(i)\) implies \((i,j)\in E\), and that \(G(E)\) is \emph{acyclic} if it contains no directed cycle.

The repair and reduction stages address three structural issues arising from raw pair predictions. In particular, directed cycles may impose unintended equalities, missing shortcut arcs may prevent the graph representation from satisfying explicit inference closure, and redundant arcs may enlarge the RMP without further contracting the dual region. Sections~\ref{subsec_cvrp_postprocessing} and~\ref{subsec_vrptw_postprocessing} separately evaluate raw prediction, graph repair, transitive reduction, and recovery to clarify the role of each component. The results show that graph repair improves bound control, whereas transitive reduction substantially compresses the deployed system while preserving the reachability relation of the repaired graph and the corresponding system of implied inequalities.

\subsection{Consistent Label Construction for Pairwise DDOIs}\label{subsec_CAPSC}

CAPSC provides an intermediate label construction strategy between relying on a single optimal dual solution and applying a conservative bound test over the entire optimal dual face. Labels derived from a single solution may be sensitive to the particular optimal dual point returned by the solver. By contrast, the bound test computes \(\underline p_i:=\min_{\mathbf p\in\mathcal D^\star}p_i\) and \(\overline p_i:=\max_{\mathbf p\in\mathcal D^\star}p_i\) for each \(i\in[m]\), and labels \(p_i\le p_j\) as positive only if \(\overline p_i\le\underline p_j\). Although this condition ensures that the inequality holds throughout the optimal dual face, computing the required bounds involves \(2m\) auxiliary optimization problems, each requiring a complete CG solve. CAPSC instead constructs a system of pairwise inequalities that is jointly supported by a common subset of sampled optimal dual solutions, thereby reducing dependence on any single dual point without requiring exhaustive optimization over the entire optimal face.

To sample informative points on \(\mathcal D^\star\), assume that the baseline CG solve has converged to \(z^\star<\infty\), let \(\bar{\mathbf p}\in\mathcal D^\star\) be the returned optimum, and fix a box radius \(M>0\). For each \(k\in[K]\), draw \(\tilde{\mathbf d}^{(k)}\sim\mathcal N(\mathbf 0,I_m)\), redrawing until \(\tilde{\mathbf d}^{(k)}\ne\mathbf 0\), set \(\mathbf d^{(k)}=\tilde{\mathbf d}^{(k)}/\|\tilde{\mathbf d}^{(k)}\|_2\), and solve
\begin{equation}\label{eq:dual_sampler}
\mathbf p^{(k)}\in
\arg\max_{\mathbf p}\!\left\{(\mathbf d^{(k)})^\top\mathbf p:
\sum_{i\in[m]}a_{ir}p_i\le c_r\ \forall r\in\Omega,\ 
\mathbf 1^\top\mathbf p\ge z^\star,\ 
\|\mathbf p-\bar{\mathbf p}\|_\infty\le M
\right\}.
\end{equation}
The dual and level constraints restrict sampling to \(\mathcal D^\star\), and the box defines a bounded, nonempty region containing \(\bar{\mathbf p}\). Problem~\eqref{eq:dual_sampler} is solved by CG through its primal form.

Given samples \(\{\mathbf p^{(k)}\}_{k=1}^K\subset\mathcal D^\star\) and tolerance \(\varepsilon>0\), let \(\mathcal V_{ij}:=\{k\in[K]:p_i^{(k)}+\varepsilon>p_j^{(k)}\}\) contain the samples that do not support \(p_i\le p_j\) with margin \(\varepsilon\). Binary variables \(\kappa_{ij}\) and \(\eta_k\) indicate whether pair \((i,j)\) is selected and sample \(k\) is retained, respectively.

For \(\alpha\in[0,1]\) with \(\lceil\alpha K\rceil\ge1\), CAPSC computes a maximum-cardinality set of pairwise inequalities among the pair and sample selections feasible for the following formulation.
\begin{subequations}\label{eq:CAPSC}
\begin{align}
\max_{\kappa,\eta}\quad
& \sum_{(i,j)\in\mathcal I}\kappa_{ij}
\label{eq:CAPSC_obj}\\
\text{s.t.}\quad
& \sum_{k\in\mathcal V_{ij}}\eta_k
   \le |\mathcal V_{ij}|(1-\kappa_{ij}),
&& \forall (i,j)\in\mathcal I,
\label{eq:CAPSC_support}\\
& \sum_{k=1}^K\eta_k \ge \lceil\alpha K\rceil,
\label{eq:CAPSC_retention}\\
& \kappa_{ij}+\kappa_{ji}\le 1,
&& \forall\,1\le i<j\le m,
\label{eq:CAPSC_antisymmetry}\\
& \kappa_{ij}\in\{0,1\},\quad \eta_k\in\{0,1\},
&& \forall (i,j)\in\mathcal I,\ k\in[K].
\label{eq:CAPSC_binary}
\end{align}
\end{subequations}
Objective~\eqref{eq:CAPSC_obj} maximizes the number of selected inequalities. Constraint~\eqref{eq:CAPSC_support} requires each selected pair to be supported by every retained sample with margin \(\varepsilon\); constraint~\eqref{eq:CAPSC_retention} enforces the retained sample threshold; and constraint~\eqref{eq:CAPSC_antisymmetry} excludes both directions of the same unordered pair. For model training, a fixed sample of ordered pairs from \(\mathcal I\) is labeled by the optimal CAPSC solution, with \(\kappa_{ij}^\star=1\) defining the positive class and \(\kappa_{ij}^\star=0\) the negative class. Proposition~\ref{prop:capsc_properties} gives the CAPSC properties used for graph repair in Section~\ref{subsec_post_processing}.

\begin{proposition}\label{prop:capsc_properties}
Assume that \(K\ge1\), \(\alpha\in[0,1]\), \(\varepsilon>0\), \(\lceil\alpha K\rceil\ge1\), and \(\mathbf p^{(k)}\in\mathcal D^\star\) for every \(k\in[K]\). For any optimal solution \((\kappa^\star,\eta^\star)\) of \eqref{eq:CAPSC}, let \(E^{\mathrm{CAPSC}}:=\{(i,j)\in\mathcal I:\kappa_{ij}^\star=1\}\), \(G^{\mathrm{CAPSC}}:=G(E^{\mathrm{CAPSC}})\), and \(\mathcal T^\star:=\{k\in[K]:\eta_k^\star=1\}\). Then \(p_i^{(k)}+\varepsilon\le p_j^{(k)}\) for every \((i,j)\in E^{\mathrm{CAPSC}}\) and \(k\in\mathcal T^\star\). Thus, \(E^{\mathrm{CAPSC}}\) is jointly deep dual-optimal, and \(G^{\mathrm{CAPSC}}\) is acyclic with inference closure.
\end{proposition}

\begin{proof}
The proof is provided in Section~\ref{sec:ec:proof_capsc_properties} of the EC appendix. \Halmos
\end{proof}

Proposition~\ref{prop:capsc_properties} shows that an optimal CAPSC solution yields a jointly deep dual-optimal label set that is acyclic and closed under inference. Optimality is required only to establish inference closure, whereas joint support and acyclicity hold for any feasible CAPSC solution. Accordingly, we solve the CAPSC integer program to optimality before generating the training labels. In the implementation, whether the sampled points lie in \(\mathcal D^\star\) and whether the pairwise support conditions hold are evaluated using the LP feasibility and objective tolerances. The resulting computational labels therefore satisfy the properties in Proposition~\ref{prop:capsc_properties} up to numerical tolerances.

\subsection{Feature Representation and Learner Instantiation for L-PDDOI Prediction}\label{subsec_learning_pair}

The CAPSC labels in Section~\ref{subsec_CAPSC} turn L-PDDOI prediction into a supervised binary classification problem. Each ordered pair \((i,j)\in\mathcal I\) receives a label indicating whether \(p_i\le p_j\) belongs to the jointly supported pairwise DDOI set. The prediction task is therefore to score candidate L-PDDOIs before the graph repair and deployment steps impose the selected inequalities jointly. We next describe the pair features and learner used in the computational study.

\paragraph{Feature design.}
The features are designed to capture routing characteristics associated with persistent differences in customer dual prices. Although dual prices are determined by the global master problem, they can reflect the marginal difficulty of covering a customer and the availability of alternative feasible routes. We use three complementary feature groups.

\begin{enumerate}[label=(\roman*),leftmargin=1.7em,itemsep=0.15em,topsep=0.25em]
    \item \emph{Customer features.} The customer group describes the two customers individually, including spatial position, demand, and local accessibility. The group captures whether one customer is intrinsically harder to cover than the other.
    \item \emph{Pair features.} The pair group characterizes the two customers jointly through proximity, relative position, neighborhood interaction, and routing compatibility. The group captures whether the two customers tend to appear in similar route alternatives or create different marginal costs.
    \item \emph{Instance features.} The instance group normalizes the same local pair pattern with respect to the overall problem scale, capacity structure, and spatial distribution.
\end{enumerate}
Section~\ref{sec:ec:pair_features} of the EC appendix provides detailed definitions of these features.

\paragraph{Learner instantiation.}
We select XGBoost as the primary learner based on validation performance among the candidate models \citep{chen2016xgboost}. Table~\ref{tab_ec_cvrp_learner_comparison} in the EC appendix reports their performance on the untouched reserved test split.

\subsection{Postprocessing and Deployment of L-PDDOIs}\label{subsec_post_processing}

Applying a threshold to independently generated pairwise scores turns prediction into a joint deployment problem. When imposed together, the selected arcs may form cycles, imply order relations that were not supported by the raw predictions, or introduce constraints that become redundant through transitivity. We therefore apply graph repair and transitive reduction before incorporating the candidate inequality system into the master problem.

Recall that Section~\ref{subsec_deployment_overview} defines the raw graph \(G^0=([m],E^0)\) and its outgoing neighborhoods \(N_{E^0}^+(i)\). Arcs in \(E^0\) represent candidate L-PDDOIs \(p_i\le p_j\), and \(E^0\subseteq\mathcal I\) excludes self loops. To repair cycles and restore inference closure, we use the pair filtering formulation (PF), which uses only the direct arcs of \(G^0\). We introduce a binary variable \(\delta_{ij}\) for each arc \((i,j)\in E^0\), where \(\delta_{ij}=1\) means that the raw arc \((i,j)\) is retained.

\noindent The repair problem is
\begin{subequations}\label{eq_postprocessing_reduced}
\begin{align}
PF:\quad \max_{\delta}\quad
& \sum_{(i,j)\in E^0}\delta_{ij}
\label{eq_postprocessing_reduced_obj}\\
\text{s.t.}\quad
& \delta_{ij}+\delta_{jk}-\delta_{ik}\le 1,
&& \forall (i,j)\in E^0,\
   k\in N_{E^0}^+(i)\cap N_{E^0}^+(j),
\label{eq_postprocessing_reduced_cons1}\\
& \delta_{ij}+\delta_{jk}\le 1,
&& \forall (i,j)\in E^0,\
   k\in N_{E^0}^+(j)\setminus N_{E^0}^+(i),
\label{eq_postprocessing_reduced_cons2}\\
& \delta_{ij}\in\{0,1\},
&& \forall (i,j)\in E^0.
\label{eq_postprocessing_reduced_cons3}
\end{align}
\end{subequations}
Objective~\eqref{eq_postprocessing_reduced_obj} maximizes the number of selected inequalities. Constraint~\eqref{eq_postprocessing_reduced_cons1} enforces closure when the shortcut \((i,k)\) is available in the raw graph, and Constraint~\eqref{eq_postprocessing_reduced_cons2} excludes two arc implications whose closure arc is unavailable. Since \(E^0\subseteq\mathcal I\), we have \(i\notin N_{E^0}^+(i)\). Constraint~\eqref{eq_postprocessing_reduced_cons2} therefore also rules out selecting both directions of the same unordered pair whenever both arcs are present. Proposition~\ref{prop_pf_feasibility} formalizes the validity of the repair formulation.

\begin{proposition}\label{prop_pf_feasibility}
For any feasible solution \(\delta\) of \eqref{eq_postprocessing_reduced}, define \(E^{\mathrm{rep}}:=\{(i,j)\in E^0:\delta_{ij}=1\}\) and \(G^{\mathrm{rep}}=G(E^{\mathrm{rep}})\). Then \(G^{\mathrm{rep}}\) is a directed acyclic graph (DAG) with inference closure.
\end{proposition}

\begin{proof}
The proof is provided in Section~\ref{sec:ec:proof_pf_feasibility} of the EC appendix. \Halmos
\end{proof}

Solving \eqref{eq_postprocessing_reduced} to optimality can be expensive when the raw prediction graph is dense. We therefore first apply a preprocessing heuristic based on strongly connected components (SCCs) to produce a feasible partial repair and shrink the remaining integer program. The heuristic decomposes \(G^0\) into SCCs, topologically orders the component DAG, and scans the induced node order in reverse. During the scan, it maintains the descendant set implied by the accepted arcs. A raw arc \((i,j)\) is accepted only when every currently implied descendant of \(j\) belongs to \(N_{E^0}^+(i)\). The test prevents \((i,j)\) from creating an implication whose required shortcut arc is unavailable in \(G^0\). Oversized SCCs are first split into smaller chunks by degree priority, so that the scan can break large cyclic blocks before applying the same closure test. The accepted arcs are fixed in the subsequent repair integer program, which optimizes over the remaining undecided arcs. Algorithm~\ref{alg:ecSccRepair} and Proposition~\ref{prop:ecSccRepair} of the EC appendix give the procedure and prove that its fixed arcs preserve PF feasibility.

Even after graph repair, the retained pairwise inequalities may contain substantial redundancy. For example, if \(p_i\le p_j\) and \(p_j\le p_k\) are both imposed, then adding \(p_i\le p_k\) does not further restrict the dual feasible region. We therefore compress the repaired graph while preserving the induced order cone. Theorem~\ref{theorem_same_feasible_region} establishes that, for pairwise systems represented by DAGs, this order cone is characterized exactly by the reachability relation of the graph.

\begin{theorem}\label{theorem_same_feasible_region}
Let \(E_1,E_2\subseteq\mathcal I\) be two arc sets whose graphs \(G(E_1)\) and \(G(E_2)\) are DAGs. Then \(\mathcal O(E_1)=\mathcal O(E_2)\) if and only if \(\mathcal R(E_1)=\mathcal R(E_2)\).
\end{theorem}

\begin{proof}
The proof is provided in Section~\ref{sec:ec:proof_same_feasible_region} of the EC appendix. \Halmos
\end{proof}

Theorem~\ref{theorem_same_feasible_region} justifies applying the standard transitive reduction to the repaired DAG \citep{aho1972transitive}. The reduction removes redundant pair inequalities while preserving the induced order cone, providing an exact compression of the deployed pair system used in CG.

\subsection{Certification and Bound Recovery of L-PDDOIs}\label{subsec_bound_recovery}

The postprocessing steps enforce graph structure, while preservation of the baseline optimal dual face requires an additional check. For each retained candidate \(e=(i,j)\), the primal master contains an L-PDDOI artificial variable \(\xi_e\ge 0\) with zero cost and column \(\mathbf b_e=\mathbf e_i-\mathbf e_j\), the primal counterpart of \(p_i-p_j\le 0\). A positive \(\xi_e\) identifies an artificial column used by the returned augmented master solution. Recovery uses this activity as a release signal tied to that solution. For any selected candidate set \(\mathcal E\), let \(z(\mathcal E)\) denote the bound obtained after full CG convergence over \(\Omega\) with the corresponding artificial variables.

Algorithm~\ref{alg:pair_recovery} states the recovery loop in exact arithmetic. BPC implementations often organize pricing as a hierarchy of routines, applying inexpensive heuristics before invoking progressively stronger pricing procedures \citep{pessoa2020generic,you2026two}. Following this practice, recovery uses \(J\) pricing stages \(\mathcal L_1,\ldots,\mathcal L_J\), with exact pricing reserved for the final stage. At round \(t\), let \(\mathcal E^{(t)}\) be the current candidate set, \(\mathcal E_{\mathrm{act}}^{(t)}:=\{e\in\mathcal E^{(t)}:\xi_e>0\}\) its active subset, and \(\mathcal E_{\mathrm{rel}}^{(t)}\) the subset released in that round. The procedure releases the active subset and, when its size does not exceed the tailing off parameter \(K_{\mathrm{tail}}\), releases all remaining candidates. Gap target recovery additionally requires a valid upper bound \(U>0\) and a target gap \(\Gamma\), where \(\operatorname{gap}(z,U):=(U-z)/U\). The gap test is applied only after the final pricing stage has completed exact pricing.

{
\DontPrintSemicolon
\begin{algorithm}[H]
\small
\caption{Recovery of candidate L-PDDOIs}\label{alg:pair_recovery}
\KwData{Initial candidate set \(\mathcal E^{(0)}\); pricing stages \(\mathcal L_1,\ldots,\mathcal L_J\); \(K_{\mathrm{tail}}\); stopping rule; \(U,\Gamma\) for gap target stopping.}
\KwResult{Retained candidate set and valid root bound.}
Set \(t\leftarrow0\), \(s\leftarrow1\)\;
\While{\(s\le J\)}{
Run CG with pair set \(\mathcal E^{(t)}\) and pricing stage \(\mathcal L_s\); obtain the current RMP value \(\widehat z^{t,s}\) and \(\{\xi_e:e\in\mathcal E^{(t)}\}\)\;
\If{gap target stopping is used, \(s=J\), and \(\operatorname{gap}(\widehat z^{t,J},U)\le\Gamma\)}{return \((\mathcal E^{(t)},\widehat z^{t,J})\)\;}
Set \(\mathcal E_{\mathrm{act}}^{(t)}\leftarrow\{e\in\mathcal E^{(t)}:\xi_e>0\}\)\;
\If{\(\mathcal E_{\mathrm{act}}^{(t)}=\varnothing\)}{
    \If{\(s=J\)}{return \((\mathcal E^{(t)},\widehat z^{t,J})\)\;}
    Set \(s\leftarrow s+1\) and continue\;
}
Set \(\mathcal E_{\mathrm{rel}}^{(t)}\leftarrow\mathcal E^{(t)}\) if \(|\mathcal E_{\mathrm{act}}^{(t)}|\le K_{\mathrm{tail}}\), and \(\mathcal E_{\mathrm{rel}}^{(t)}\leftarrow\mathcal E_{\mathrm{act}}^{(t)}\) otherwise\;
Set \(\mathcal E^{(t+1)}\leftarrow\mathcal E^{(t)}\setminus\mathcal E_{\mathrm{rel}}^{(t)}\), reset the pricing state, and set \(t\leftarrow t+1\)\;
}
\end{algorithm}
}

When \(J=1\), recovery completes CG with exact pricing before activity is inspected in every round. For \(J>1\), recovery applies the release loop through stronger pricing stages. In either case, the objective at an intermediate heuristic stage is an RMP value, so gap stopping is deferred. At the final stage, full exact pricing and exact solution of the augmented master give \(\widehat z^{t,J}=z(\mathcal E^{(t)})\), so the gap test is valid. Without gap target stopping, zero activity in every pair artificial variable at this stage gives the condition used in Proposition~\ref{prop_bound_recovery_certificate}.

The implementation replaces the exact test \(\xi_e>0\) by \(\xi_e>\epsilon_{\mathrm{act}}\) to accommodate floating point LP solutions. Under this test, an empty numerical active set gives \(0\le\xi_e\le\epsilon_{\mathrm{act}}\), whereas Proposition~\ref{prop_bound_recovery_certificate} assumes \(\xi_e=0\). The numerical experiments verify the recovered objective against the independently recorded baseline CG bound after exact pricing.

\begin{proposition}\label{prop_bound_recovery_certificate}
Assume that the original master LP and its augmentation by \(\mathcal E\) attain finite optimal solutions. For any \(\mathcal E'\subseteq\mathcal E\), \(z(\mathcal E')\ge z(\mathcal E)\). Moreover, \textnormal{(i)} \(\mathcal E\) is jointly deep dual-optimal, \textnormal{(ii)} \(z(\mathcal E)=z^\star\), and \textnormal{(iii)} the augmented master has an optimal solution satisfying \(\xi_e=0\) for every \(e\in\mathcal E\) are equivalent. When these conditions hold, every optimal solution of the augmented dual belongs to \(\mathcal D^\star\), and \(\mathcal E\) is a certified L-PDDOI set.
\end{proposition}

\begin{proof}
The proof is provided in Section~\ref{sec:ec:proof_bound_recovery_certificate} of the EC appendix. \Halmos
\end{proof}

Corollary~\ref{cor:recoveryTermination} in the EC appendix shows that exact recovery without gap target stopping terminates after at most \(|\mathcal E^{(0)}|\) release rounds and returns the baseline CG bound.

\section{Numerical Study}\label{sec_exp_design}

The numerical study evaluates L-PDDOIs for stabilizing root CG on CVRP and VRPTW instances, where \(n\) denotes the number of customers. All experiments use RouteOpt, an exact BPC solver for both problems \citep{you2026routeopt}. Section~\ref{subsec_protocol} presents the experimental protocol, followed by the CVRP and VRPTW results in Sections~\ref{sec_cvrp} and~\ref{subsec_vrptw_case}, respectively.

\subsection{Experimental Protocol}\label{subsec_protocol}

\subsubsection{Implementation and Computing Environment}\label{subsubsec_implementation}

All experiments are conducted on an Ubuntu 24.04.4 workstation equipped with an AMD Ryzen 9 9950X3D processor and 60~GiB of memory. Gurobi 13.0.1 \citep{gurobi} is used to solve all linear and mixed-integer programs in this paper. Each timed run uses a single Gurobi thread and no time limit.

\subsubsection{Instance Sets and Learning Splits}\label{subsubsec_instances}

Table~\ref{tab_protocol_instances} summarizes the seven instance sets and their roles in the computational study. For the CVRP, XML500-Learn is used for training and model selection, XML500-Test for solver evaluation, and the CVRPLIB X-series (X100) for transfer evaluation. The XML500 instances are generated using the official CVRPLIB generator \citep{queiroga202110000}; Section~\ref{sec:ec:xml500_generation} of the EC appendix describes the generation procedure and profile matching in detail. For the VRPTW, the synthetic 200-customer Gehring--Homberger sets S-GH200-Learn and S-GH200-Test are used for training and validation, and for prediction reporting and solver evaluation, respectively. Separate official Gehring--Homberger sets with 200 and 400 customers are used for transfer evaluation \citep{gehring2002parallelization}. Section~\ref{appendix_gh200_generation} of the EC appendix describes the synthetic instance generator.
\begin{table}[htbp]
\centering
\caption{Instance sets and their roles in the numerical study.}
\label{tab_protocol_instances}
\vspace{-3pt}
\footnotesize
\setlength{\tabcolsep}{7pt}
\renewcommand{\arraystretch}{1.08}
\begin{tabular}{@{}llrrll@{}}
\toprule
Problem & Instance set & \#Inst. & \#Customers & Source & Role \\
\midrule
CVRP  & XML500-Learn    & 500 & 500       & Generated (XML)             & Training and model selection \\
CVRP  & XML500-Test     & 200 & 500       & Generated (XML)             & Main solver evaluation \\
CVRP  & X100            & 100 & 100--1000 & CVRPLIB X-series            & Transfer evaluation \\
\midrule
VRPTW & S-GH200-Learn   & 300 & 200       & Generated (S-GH200)         & Training and validation \\
VRPTW & S-GH200-Test    &  60 & 200       & Generated (S-GH200)         & Prediction and solver evaluation \\
VRPTW & GH200-Benchmark &  30 & 200       & Official Gehring--Homberger & Transfer evaluation \\
VRPTW & GH400-Benchmark &  30 & 400       & Official Gehring--Homberger & Size transfer evaluation \\
\bottomrule
\end{tabular}
\vspace{-5pt}
\end{table}

\subsubsection{Methods, Metrics, and Reporting Conventions}\label{subsubsec_methods}

The four primary root CG configurations are the proposed \emph{L-PDDOI} and \emph{L-PDDOI-Rec} methods, the unstabilized \emph{Default} configuration, and the \emph{Auto-Stab} scheme of \citet{pessoa2018automation}. L-PDDOI directly deploys the repaired and reduced candidate set without certification, whereas L-PDDOI-Rec applies numerical recovery and verifies the recovered bound against the baseline. The CVRP study further includes two ablation variants that replace the learned L-PDDOIs with either dual-value predictions obtained by regression (\emph{Dual-Reg}) or an equal number of randomly selected pairs consistent with an acyclic ordering (\emph{Rand-pair}). All reported root CG experiments use the same \(ng\)-route relaxation, initial RMP, pricing routines, reduced-cost tolerance, and column-management policy. Cut separation and the vehicle-count constraint are disabled.

We distinguish the time spent in root CG from the overhead associated with model deployment. Specifically, \(T_{\mathrm{CG}}\) measures the elapsed time from the first RMP solve through the completion of exact pricing, after the candidate set has been constructed. The measure includes pricing time \(T_{\mathrm{price}}\) and LP reoptimization time \(T_{\mathrm{LP}}\). Prediction time \(T_{\mathrm{pred}}\) includes feature construction and model inference, while \(T_{\mathrm{post}}\) records postprocessing. Total runtime is \(T_{\mathrm{pred}}+T_{\mathrm{post}}+T_{\mathrm{CG}}\).

The remaining measures are organized into three categories. For an evaluated candidate set \(\mathcal E\), \emph{solution quality metrics} include normalized bound loss \(\Delta_{\mathrm{bd}}(\mathcal E):=[z^\star-z(\mathcal E)]/\max\{|z^\star|,1\}\) and the gap to the best known solution (BKS), \(\operatorname{gap}_{\mathrm{BKS}}(\mathcal E):=[U_{\mathrm{BKS}}-z(\mathcal E)]/U_{\mathrm{BKS}}\), where \(U_{\mathrm{BKS}}\) is the BKS value. \emph{Workload metrics} record the number of pricing iterations (\#Iter.), route columns at termination (\#Col.), recovery rounds (\#Rec.), and active L-PDDOI artificial columns (APCs) at termination (\#Act.~APC). An APC is numerically active when \(\xi_e>\epsilon_{\mathrm{act}}\). Finally, \emph{deployment metrics} describe the number of raw arcs (\#Raw~arcs), the size of the largest strongly connected component (Max.~SCC), and the number of candidate L-PDDOIs deployed initially (\#Init.~L-PDDOIs).

Unless otherwise specified in a table note, runtimes, iteration counts, column counts, raw arc counts, and deployed L-PDDOI counts are reported as geometric means because their values span several orders of magnitude. Bound loss, optimality gap, Max.~SCC, and the number of recovery rounds are reported as arithmetic means. Because the number of active L-PDDOI artificial columns can be zero, \#Act.~APC uses the shifted geometric mean obtained by adding one before aggregation and subtracting one afterward. Time reduction is computed as \(100\bigl[1-\operatorname{GM}_i(T_i/T_i^{\mathrm{Default}})\bigr]\), where \(\operatorname{GM}_i\) is the geometric mean over paired instances.

\subsubsection{Parameter Settings and Reproducibility}\label{subsubsec_parameters}

The dual sampler uses \(M=10^6\), and CAPSC uses \(K=20\), \(\alpha=0.8\), and \(\varepsilon=10^{-6}\). The deployment cutoff was fixed at \(\tau_{\mathrm{pred}}=0.5\) before any solver-level evaluation and was used in all primary experiments. XML500-Test and X100 played no role in fixing this value. The threshold sweep in Section~\ref{subsec_cvrp_frontier} is reported solely as a post hoc sensitivity analysis. The settings specific to the CVRP and VRPTW, including those for Auto-Stab and recovery, are reported in Sections~\ref{sec_cvrp} and~\ref{subsec_vrptw_case}, respectively.

The XGBoost pair predictor \citep{chen2016xgboost} is fitted using the training instances. Prediction quality is evaluated on test instances excluded from fitting and validation, and all pairs associated with the same instance remain in one split. XML500-Test is reserved for solver evaluation, whereas S-GH200-Test is used for both prediction reporting and solver evaluation. Sections~\ref{appendix_cvrp_learning_training} and~\ref{sec:ec:vrptw_features} of the EC appendix give the exact data splits and learning protocols.

For reproducibility and reuse, we will release the instance generators, feature construction and model training code, solver integration, records for each instance, and the complete selected XGBoost parameters in a public repository after acceptance.

\subsection{Root Node Stabilization for the CVRP}\label{sec_cvrp}

The CVRP has deterministic customer demands, one depot, homogeneous vehicle capacity, and pairwise travel costs. Each route starts and ends at the depot and serves customers within capacity. A feasible solution covers every customer once.

For the CVRP, Auto-Stab uses the RouteOpt default settings, with penalty coefficient $\gamma=0.12$, delta scale $1$, and initial extra scale $20$. L-PDDOI-Rec uses three pricing stages ($J=3$), consisting of light heuristic pricing, heavy heuristic pricing, and exact pricing. It sets $K_{\mathrm{tail}}=20$ and $\epsilon_{\mathrm{act}}=10^{-3}$, and stops after exact pricing when no L-PDDOI artificial variable exceeds the activity tolerance. The recovered bound is then checked against the baseline CG bound.

The CVRP learner comparison uses a common representation with 83 candidate features. On the reserved test split of XML500-Learn, XGBoost attains an average precision (AP) of $0.9913$, where $\mathrm{AP}:=\sum_k(R_k-R_{k-1})P_k$ and $P_k$ and $R_k$ denote precision and recall at the $k$th score threshold. It also attains an F1 score of $0.9498$, where $\mathrm{F1}:=2PR/(P+R)$ and $P$ and $R$ denote precision and recall for the predicted labels. Both values are the highest among the four candidate learners. The final deployment predictor is trained separately using the 32 features retained after screening and attains an AP of $0.9909$ and an F1 score of $0.9484$ on the same reserved test split.

\subsubsection{Main Computational Results}\label{subsec_cvrp_main}

Table~\ref{tab_cvrp_main} compares the six root CG configurations defined in Section~\ref{subsubsec_methods} on XML500-Test. The comparison evaluates overall computational performance, the cost of preserving the baseline CG bound through recovery, and the contribution of the learned L-PDDOIs beyond contraction of the dual space.

\begin{table}[htbp]
\centering
\caption{Root column generation performance on XML500-Test.}
\label{tab_cvrp_main}
\vspace{-3pt}
\footnotesize
\setlength{\tabcolsep}{2.0pt}
\renewcommand{\arraystretch}{1.08}
\begin{tabular}{@{}lrrrrrrrrrrr@{}}
\toprule
& \multicolumn{3}{c}{CG time} 
& \multicolumn{2}{c}{Overhead} 
& \multicolumn{2}{c}{Performance} 
& \multicolumn{1}{c}{Deployment}
& \multicolumn{2}{c}{CG work} 
& \multicolumn{1}{c}{Recovery} \\
\cmidrule(lr){2-4}
\cmidrule(lr){5-6}
\cmidrule(lr){7-8}
\cmidrule(lr){9-9}
\cmidrule(lr){10-11}
\cmidrule(l){12-12}
Method 
& \(T_{\mathrm{price}}\)/s 
& \(T_{\mathrm{LP}}\)/s 
& \(T_{\mathrm{CG}}\)/s 
& Pred./s 
& Post./s 
& Red./\% 
& \(\Delta_{\mathrm{bd}}/\%\)
& \makecell{\#Init.\\L-PDDOIs}
& \#Iter. 
& \#Col. 
& \#Rec. \\
\midrule
Default     
& 79.4 & 36.3 & 119.9 & -- & -- & --   & --  & --    & 733 & 24,548 & --  \\
Auto-Stab   
& 69.1 & 31.5 & 103.4 & -- & -- & 13.7 & --  & --    & 573 & 20,662 & --  \\
\addlinespace[2pt]
L-PDDOI
& 10.4 & 1.6 & 12.3 & 0.19 & 0.26 & 89.7 & 1.3 & 2,045 & 64 & 4,689 & -- \\
L-PDDOI-Rec
& 44.4 & 9.0 & 54.2 & 0.19 & 0.26 & 54.8 & 0.0 & 2,045 & 326 & 10,601 & 8.8 \\
\addlinespace[2pt]
Dual-Reg
& 6.2 & 0.8 & 7.2 & 0.01 & -- & 94.0 & 7.8 & 499 & 40 & 1,984 & -- \\
Rand-pair
& 13.2 & 1.8 & 15.5 & -- & -- & 87.1 & 48.0 & 2,045 & 115 & 7,743 & -- \\
\bottomrule
\end{tabular}
\vspace{-5pt}
\end{table}

\paragraph{Direct deployment and recovery.}
At $\tau_{\mathrm{pred}}=0.5$, direct L-PDDOI reduces the geometric mean root CG time from $119.9$ to $12.3$ seconds, an $89.7\%$ reduction. Pricing and LP times fall from $79.4$ and $36.3$ seconds to $10.4$ and $1.6$ seconds, and pricing iterations from $733$ to $64$. Feature construction and prediction add $0.19$ seconds in total, while postprocessing adds $0.26$ seconds. The mean bound loss is $1.3\%$.
L-PDDOI-Rec matches the recorded baseline CG bound and reduces the geometric mean root CG time to $54.2$ seconds, a $54.8\%$ reduction, with an average of $8.8$ recovery rounds.
Direct deployment provides greater time savings, whereas recovery restores the recorded baseline CG bound. Section~\ref{subsec_cvrp_postprocessing} examines recovery cost.

\paragraph{Comparison with the stabilization baseline.}
Auto-Stab also preserves the baseline CG bound and reduces the geometric mean root CG time to $103.4$ seconds, a $13.7\%$ reduction. Under the same bound requirement, L-PDDOI-Rec reaches $54.2$ seconds and a $54.8\%$ reduction. Figure~\ref{fig_cvrp_scatter_main} complements these aggregate results with comparisons by instance against Default. Both L-PDDOI and L-PDDOI-Rec are faster than Default on all $200$ XML500-Test instances, so the aggregate reductions are not driven by a small subset of instances. The separation between the two point clouds shows the additional time required for bound recovery.

\begin{figure}[htbp]
\centering
\caption{Root CG time by instance relative to Default on XML500-Test. Points below the diagonal are faster.}
\label{fig_cvrp_scatter_main}
\vspace{2pt}
\includegraphics[width=0.54\linewidth]{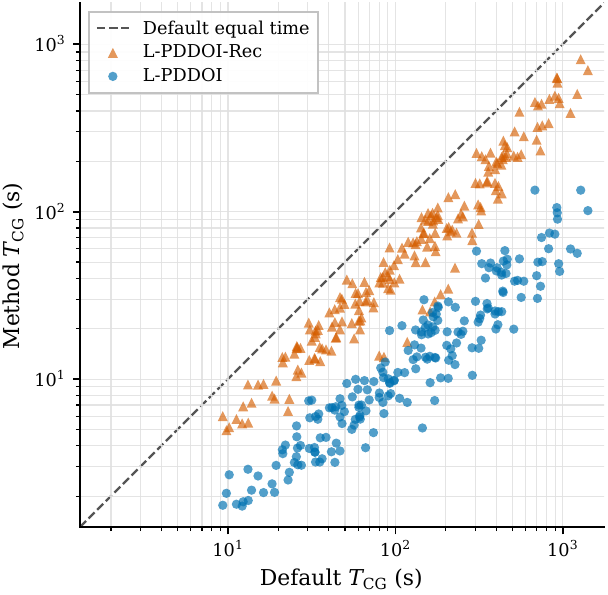}
\vspace{-5pt}
\end{figure}

\paragraph{Comparison with the ablation baselines.}
Dual-Reg replaces the pair classifier with regressed dual values and achieves a $94.0\%$ time reduction with a mean bound loss of $7.8\%$. Rand-pair draws the same number of candidate inequalities from a random acyclic ordering and achieves an $87.1\%$ time reduction with a mean bound loss of $48.0\%$. Both ablations accelerate CG by contracting the dual region, but their bound losses exceed the $1.3\%$ obtained by L-PDDOI. The comparison indicates that contraction drives the computational improvement, whereas the learned L-PDDOIs improve compatibility with the optimal dual face.

\subsubsection{Threshold Effects on Time Reduction and Bound Loss}\label{subsec_cvrp_frontier}

The prediction threshold $\tau_{\mathrm{pred}}$ controls how aggressively candidate L-PDDOIs are selected. A lower threshold admits more pairs, which can strengthen stabilization but also increase the risk of bound loss. Figure~\ref{fig_cvrp_time_reduction_loss_frontier} illustrates the resulting tradeoff on XML500-Test. As $\tau_{\mathrm{pred}}$ increases from $0.5$ to $0.9$, the time reduction decreases from $89.7\%$ to $71.4\%$, while the mean bound loss falls from $1.3\%$ to $0.2\%$. The intermediate thresholds follow the same monotone pattern.

\begin{figure}[htbp]
\centering
\caption{Time reduction and mean bound loss across deployment thresholds on XML500-Test.}
\label{fig_cvrp_time_reduction_loss_frontier}
\vspace{2pt}
\includegraphics[width=0.82\linewidth]{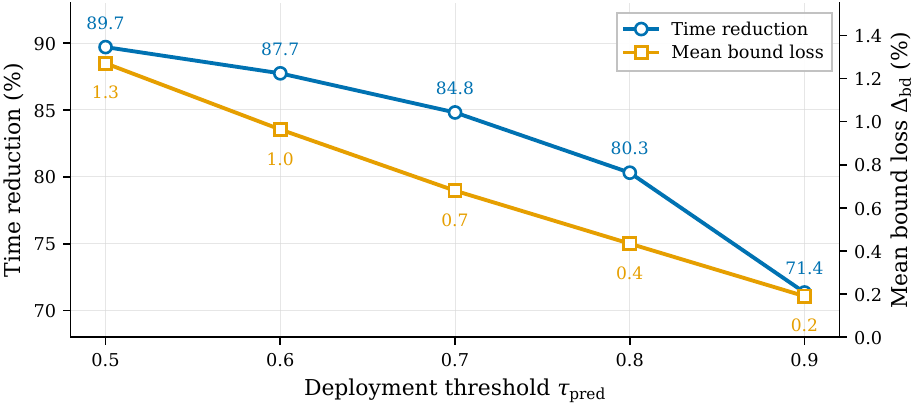}
\vspace{-5pt}
\end{figure}

\paragraph{Effect on the predicted graph.}
Table~\ref{tab_cvrp_threshold_sensitivity} shows how the prediction threshold affects the raw graph and deployed L-PDDOI set. As \(\tau_{\mathrm{pred}}\) decreases from \(0.9\) to \(0.5\), the raw arc count rises from \(100{,}519\) to \(119{,}794\), while the mean size of the largest SCC increases from \(1.0\) to \(157.6\) nodes. Despite this greater density, graph repair and transitive reduction reduce the deployed set from \(10{,}649\) to \(2{,}045\) L-PDDOIs and LP time per iteration from \(0.048\) to \(0.025\) seconds because denser prediction graphs contain more redundant arcs that can be removed.

\begin{table}[H]
\centering
\caption{Threshold deployment diagnostics on XML500-Test.}
\label{tab_cvrp_threshold_sensitivity}
\vspace{-3pt}
\footnotesize
\setlength{\tabcolsep}{7pt}
\renewcommand{\arraystretch}{1.08}
\begin{tabular}{@{}lrrrrr@{}}
\toprule
Metric & $\tau_{\mathrm{pred}}=0.90$ & $0.80$ & $0.70$ & $0.60$ & $0.50$ \\
\midrule
\#Raw arcs & 100,519 & 108,028 & 112,734 & 116,460 & 119,794 \\
Max. SCC & 1.0 & 4.1 & 15.2 & 57.7 & 157.6 \\
\#Init. L-PDDOIs & 10,649 & 6,776 & 4,571 & 3,100 & 2,045 \\
$T_{\mathrm{LP}}/\#\mathrm{Iter.}$/s & 0.048 & 0.038 & 0.033 & 0.029 & 0.025 \\
\bottomrule
\end{tabular}
\vspace{-5pt}
\end{table}

\subsubsection{Postprocessing and Bound Recovery}\label{subsec_cvrp_postprocessing}

Table~\ref{tab_cvrp_postprocessing} isolates graph repair, transitive reduction, and recovery at $\tau_{\mathrm{pred}}=0.5$. Raw Prediction reduces root CG time to $10.1$ seconds but has mean and maximum bound losses of $9.5\%$ and $66.2\%$. Repair Only reduces these losses to $1.3\%$ and $9.7\%$, respectively, while increasing root CG time to $18.0$ seconds. Consistent with Theorem~\ref{theorem_same_feasible_region}, transitive reduction preserves the same bound loss profile while reducing the candidate set from $117{,}457$ to $2{,}045$ inequalities. The Full pipeline then lowers LP time from $4.7$ to $1.6$ seconds and root CG time from $18.0$ to $12.3$ seconds.

The sharp decrease in bound loss after repair suggests that deployment error is driven not only by individual false-positive arcs, but also by the additional order relations created when predictions are composed. Transitive reduction addresses a different source of cost. It leaves the reachability relation and order cone unchanged but removes constraints that would otherwise enlarge every RMP solve and increase recovery overhead.

The recovery rows quantify the benefit of applying recovery to the compact inequality set. Starting from the full postprocessing pipeline reduces the average number of recovery rounds from $56.3$ to $8.8$ and the geometric mean root CG time from $110.6$ to $54.2$ seconds.

\begin{table}[b]
\centering
\caption{Ablation of postprocessing and recovery on XML500-Test.}
\label{tab_cvrp_postprocessing}
\vspace{-3pt}
\footnotesize
\setlength{\tabcolsep}{4.0pt}
\renewcommand{\arraystretch}{1.08}
\begin{tabular}{@{}lrrrrrrr@{}}
\toprule
& \multicolumn{1}{c}{Deployment}
& \multicolumn{2}{c}{CG time}
& \multicolumn{2}{c}{Bound loss}
& \multicolumn{2}{c}{Recovery} \\
\cmidrule(lr){2-2}
\cmidrule(lr){3-4}
\cmidrule(lr){5-6}
\cmidrule(l){7-8}
Variant
& \#Init. L-PDDOIs
& \(T_{\mathrm{CG}}\)/s
& \(T_{\mathrm{LP}}\)/s
& Mean \(\Delta_{\mathrm{bd}}\)/\%
& Max. \(\Delta_{\mathrm{bd}}\)/\%
& \#Rec.
& \#Act. APC \\
\midrule
Default         & --      & 119.9 & 36.3 & --   & --   & --   & -- \\
\midrule
Raw Prediction  & 119,794 & 10.1  & 2.1  & 9.5 & 66.2 & --   & 365 \\
Repair Only     & 117,457 & 18.0  & 4.7  & 1.3 & 9.7  & --   & 270 \\
Full            & 2,045   & 12.3  & 1.6  & 1.3 & 9.7  & --   & 276 \\
\midrule
Raw+Rec         & 119,794 & 110.6 & 37.0 & 0.0 & 0.0  & 56.3 & 0 \\
Full+Rec        & 2,045   & 54.2  & 9.0  & 0.0 & 0.0  & 8.8 & 0 \\
\bottomrule
\end{tabular}
\vspace{-5pt}
\end{table}

\subsubsection{Transfer to the CVRPLIB X-Series Benchmark}\label{subsec_cvrp_transfer}

We evaluate the model trained on XML500-Learn directly on X100 without retraining. Table~\ref{tab_cvrp_generalization} reports the aggregate result and partitions X100 by customer count.

\begin{table}[htbp]
\centering
\caption{L-PDDOI transfer performance on CVRPLIB X-series instances.}
\label{tab_cvrp_generalization}
\vspace{-3pt}
\footnotesize
\setlength{\tabcolsep}{3pt}
\renewcommand{\arraystretch}{1.08}
\begin{tabular}{@{}lrrrrrrrrr@{}}
\toprule
& \multicolumn{2}{c}{Instance set}
& \multicolumn{1}{c}{Deployment}
& \multicolumn{2}{c}{CG time}
& \multicolumn{4}{c}{Performance} \\
\cmidrule(lr){2-3}
\cmidrule(lr){4-4}
\cmidrule(lr){5-6}
\cmidrule(l){7-10}
Instance set
& \#Inst.
& Avg. \(n\)
& \makecell{\#Init.\\L-PDDOIs}
& \makecell{Default\\\(T_{\mathrm{CG}}\)/s}
& \makecell{L-PDDOI\\\(T_{\mathrm{CG}}\)/s}
& Red./\%
& \(\Delta_{\mathrm{bd}}/\%\)
& BKS gap/\%
& \makecell{\#BKS gap\\\(<5\%\)} \\
\midrule
XML500-Test & 200 & 500 & 2,045 & 119.9 & 12.3 & 89.7 & 1.3 & --  & -- \\
X100        & 100 & 412 & 1,312 & 40.7  & 5.9  & 85.6 & 1.1 & 2.4 & 97 \\
\quad \(n<300\)        & 43 & 199 &   623 &  6.8  & 1.6  & 76.2 & 1.1 & 2.7 & 42 \\
\quad \(300\le n<600\) & 34 & 426 & 1,646 & 59.3  & 7.4  & 87.4 & 1.2 & 2.3 & 32 \\
\quad \(n\ge600\)      & 23 & 791 & 3,781 & 656.3 & 45.6 & 93.1 & 1.1 & 1.9 & 23 \\
\bottomrule
\end{tabular}
\vspace{-5pt}
\end{table}

L-PDDOI reduces the geometric mean root CG time on X100 from $40.7$ to $5.9$ seconds, an $85.6\%$ reduction, with a mean bound loss of $1.1\%$ and a mean BKS gap of $2.4\%$. The BKS gap is below $5.0\%$ for $97$ of the $100$ instances and for all $23$ instances with $n\ge 600$. Across the three size groups, the time reduction increases from $76.2\%$ for $n<300$ to $93.1\%$ for $n\ge 600$, while the mean bound loss remains between $1.1\%$ and $1.2\%$ and the mean BKS gap decreases from $2.7\%$ to $1.9\%$. The transfer results indicate that the learned L-PDDOIs remain effective across the customer ranges represented in the X-series benchmark.

\subsection{Root Node Stabilization for the VRPTW}\label{subsec_vrptw_case}

We next examine L-PDDOIs on the VRPTW, where time windows couple routing and scheduling decisions, expand the pricing resource space, and change which customers are difficult to cover. Pairwise inequalities that remain stable in a spatial setting may lose their stability when temporal constraints bind, making the VRPTW a more stringent test.

For the VRPTW, Auto-Stab uses the RouteOpt defaults, with \(\gamma=0.3\), delta scale \(1\), and initial extra scale \(10\). L-PDDOI-Rec uses one exact pricing stage (\(J=1\)) with \(K_{\mathrm{tail}}=0\) and \(\epsilon_{\mathrm{act}}=10^{-3}\).

The VRPTW pair predictor uses 59 features, consisting of the 32 CVRP features, 24 time window features, and three spatial regime indicators. Its AP and F1 score on S-GH200-Test are $0.9709$ and $0.9033$, respectively. S-GH200-Test is excluded from fitting and validation and serves as the solver evaluation set. Section~\ref{sec:ec:vrptw_features} of the EC appendix reports the features and protocol.

\subsubsection{Main Computational Results}\label{subsec_vrptw_main}

Table~\ref{tab_vrptw_main} compares the four primary root CG configurations on S-GH200-Test across runtime, workload, and bound quality.

\begin{table}[H]
\centering
\caption{Root column generation performance on S-GH200-Test.}
\label{tab_vrptw_main}
\vspace{-3pt}
\footnotesize
\setlength{\tabcolsep}{2.0pt}
\renewcommand{\arraystretch}{1.08}
\begin{tabular}{@{}lrrrrrrrrrrr@{}}
\toprule
& \multicolumn{3}{c}{CG time}
& \multicolumn{2}{c}{Overhead}
& \multicolumn{2}{c}{Performance}
& \multicolumn{1}{c}{Deployment}
& \multicolumn{2}{c}{CG work}
& \multicolumn{1}{c}{Recovery} \\
\cmidrule(lr){2-4}
\cmidrule(lr){5-6}
\cmidrule(lr){7-8}
\cmidrule(lr){9-9}
\cmidrule(lr){10-11}
\cmidrule(l){12-12}
Method
& \(T_{\mathrm{price}}\)/s
& \(T_{\mathrm{LP}}\)/s
& \(T_{\mathrm{CG}}\)/s
& Pred./s
& Post./s
& Red./\%
& \(\Delta_{\mathrm{bd}}\)/\%
& \makecell{\#Init.\\L-PDDOIs}
& \#Iter.
& \#Col.
& \#Rec. \\
\midrule
Default
& 24.1 & 64.0 & 95.2 & -- & -- & -- & -- & -- & 24,977 & 89,087 & -- \\
Auto-Stab
& 15.7 & 24.9 & 45.7 & -- & -- & 52.0 & -- & -- & 6,557 & 71,755 & -- \\
\addlinespace[2pt]
L-PDDOI
& 3.2 & 1.5 & 5.8 & 2.72 & 0.03 & 93.9 & 0.5 & 1,417 & 376 & 13,803 & -- \\
L-PDDOI-Rec
& 8.2 & 5.3 & 16.1 & 2.72 & 0.03 & 83.1 & 0.0 & 1,417 & 544 & 18,428 & 8.1 \\
\bottomrule
\end{tabular}
\vspace{-5pt}
\end{table}

\paragraph{Direct deployment and recovery.}
With $\tau_{\mathrm{pred}}=0.5$, direct L-PDDOI deployment reduces the geometric mean root CG time from $95.2$ to $5.8$ seconds, a $93.9\%$ reduction. Pricing time decreases from $24.1$ to $3.2$ seconds, LP reoptimization time from $64.0$ to $1.5$ seconds, and the number of pricing iterations from $24{,}977$ to $376$. The geometric mean time for feature construction and prediction is $2.72$ seconds, while postprocessing takes $0.03$ seconds, and the mean bound loss is $0.5\%$. L-PDDOI-Rec matches the recorded baseline CG bound and reduces root CG time to $16.1$ seconds, an $83.1\%$ reduction, with an average of $8.1$ recovery rounds. After accounting for prediction overhead and postprocessing, L-PDDOI and L-PDDOI-Rec reduce the geometric mean total runtime on S-GH200-Test by \(90.6\%\) and \(79.4\%\), respectively.

\subsubsection{Postprocessing and Bound Recovery}\label{subsec_vrptw_postprocessing}

Table~\ref{tab_vrptw_postprocessing} isolates graph repair, transitive reduction, and recovery at $\tau_{\mathrm{pred}}=0.5$. Raw Prediction reduces root CG time to $2.4$ seconds but has mean and maximum bound losses of $15.0\%$ and $40.9\%$. Repair Only reduces these losses to $0.5\%$ and $1.7\%$, respectively, while changing root CG time to $6.4$ seconds. Consistent with Theorem~\ref{theorem_same_feasible_region}, transitive reduction preserves the same bound loss profile while reducing the candidate set from $11{,}411$ to $1{,}417$ inequalities. The Full pipeline then lowers LP time from $2.0$ to $1.5$ seconds and root CG time from $6.4$ to $5.8$ seconds. Relative to Raw+Rec, the Full pipeline reduces the average number of recovery rounds from $34.4$ to $8.1$ and the geometric mean root CG time from $43.5$ to $16.1$ seconds. Together with the CVRP results, these findings show that graph repair, transitive reduction, and recovery serve complementary roles across routing problems with different feasibility structures.

\begin{table}[H]
\centering
\caption{Ablation of postprocessing and recovery on S-GH200-Test.}
\label{tab_vrptw_postprocessing}
\vspace{-3pt}
\footnotesize
\setlength{\tabcolsep}{4.0pt}
\renewcommand{\arraystretch}{1.08}
\begin{tabular}{@{}lrrrrrrr@{}}
\toprule
& \multicolumn{1}{c}{Deployment}
& \multicolumn{2}{c}{CG time}
& \multicolumn{2}{c}{Bound loss}
& \multicolumn{2}{c}{Recovery} \\
\cmidrule(lr){2-2}
\cmidrule(lr){3-4}
\cmidrule(lr){5-6}
\cmidrule(l){7-8}
Variant
& \#Init. L-PDDOIs
& \(T_{\mathrm{CG}}\)/s
& \(T_{\mathrm{LP}}\)/s
& Mean \(\Delta_{\mathrm{bd}}\)/\%
& Max. \(\Delta_{\mathrm{bd}}\)/\%
& \#Rec.
& \#Act. APC \\
\midrule
Default         & --     & 95.2 & 64.0 & --   & --   & --   & -- \\
\midrule
Raw Prediction  & 12,620 & 2.4  & 0.3  & 15.0 & 40.9 & --   & 112 \\
Repair Only     & 11,411 & 6.4  & 2.0  & 0.5  & 1.7  & --   & 39 \\
Full            & 1,417  & 5.8  & 1.5  & 0.5  & 1.7  & --   & 40 \\
\midrule
Raw+Rec         & 12,620 & 43.5 & 17.2 & 0.0  & 0.0  & 34.4 & 0 \\
Full+Rec        & 1,417  & 16.1 & 5.3  & 0.0  & 0.0  & 8.1  & 0 \\
\bottomrule
\end{tabular}
\vspace{-5pt}
\end{table}

\subsubsection{Transfer to the Official Gehring--Homberger Benchmarks}\label{subsec_vrptw_transfer}

We transfer the model trained on S-GH200-Learn without retraining to the official 200- and 400-customer Gehring--Homberger instances. Table~\ref{tab_vrptw_benchmark} reports the aggregate results.

\begin{table}[H]
\centering
\caption{L-PDDOI transfer performance on Gehring--Homberger instances.}
\label{tab_vrptw_benchmark}
\vspace{-3pt}
\footnotesize
\setlength{\tabcolsep}{3.0pt}
\renewcommand{\arraystretch}{1.08}
\begin{tabular}{@{}llrrrrrrrrr@{}}
\toprule
& & \multicolumn{2}{c}{Instance set}
& \multicolumn{1}{c}{Deployment}
& \multicolumn{2}{c}{CG time}
& \multicolumn{4}{c}{Performance} \\
\cmidrule(lr){3-4}
\cmidrule(lr){5-5}
\cmidrule(lr){6-7}
\cmidrule(l){8-11}
Instance set
& Method
& \#Inst.
& Avg. \(n\)
& \makecell{\#Init.\\L-PDDOIs}
& \makecell{Default\\\(T_{\mathrm{CG}}\)/s}
& \makecell{Method\\\(T_{\mathrm{CG}}\)/s}
& Red./\%
& \(\Delta_{\mathrm{bd}}\)/\%
& BKS gap/\%
& \makecell{\#BKS gap\\\(<5\%\)} \\
\midrule
GH200-Benchmark & L-PDDOI      & 30 & 200 & 1,440 & 111.0 & 6.0  & 94.6 & 0.4 & 2.5 & 28 \\
GH200-Benchmark & L-PDDOI-Rec & 30 & 200 & 1,440 & 111.0 & 14.8 & 86.7 & 0.0 & 2.0 & 29 \\
\midrule
GH400-Benchmark & L-PDDOI      & 30 & 400 & 5,428 & 1,176.7 & 27.0  & 97.7 & 6.0 & 8.9 & 9 \\
GH400-Benchmark & L-PDDOI-Rec & 30 & 400 & 5,428 & 1,176.7 & 272.6 & 76.8 & 0.0 & 3.2 & 23 \\
\bottomrule
\end{tabular}
\vspace{-5pt}
\end{table}

On GH200-Benchmark, direct deployment reduces root CG time by $94.6\%$ with a mean bound loss of $0.4\%$; recovery retains an $86.7\%$ reduction and lowers the mean BKS gap from $2.5\%$ to $2.0\%$. On GH400-Benchmark, direct deployment retains a $97.7\%$ reduction, but mean bound loss rises to $6.0\%$, and only $9$ of $30$ instances remain below a $5.0\%$ BKS gap. Recovery retains a $76.8\%$ reduction, lowers the gap from $8.9\%$ to $3.2\%$, and increases this count to $23$.

The transfer pattern reflects a structural distribution shift because the transition from GH200 to GH400 changes not only the number of customers but also the coordinate layout and the distributions of neighborhoods and routing compatibility. More specifically, within each Gehring--Homberger benchmark size and class, including R2, C2, and RC2, instances share a common customer coordinate layout and differ primarily in their time windows. Training on these instances therefore exposes the model to substantial temporal variation but limited diversity in spatial layouts and graph structures. When the instance size changes, the model encounters a new coordinate layout together with different distributions of distances, neighborhoods, and routing compatibility that are not well represented in the training data. Limited exposure to graph heterogeneity may explain why the mean bound loss increases from \(0.4\%\) on GH200-Benchmark to \(6.0\%\) on GH400-Benchmark. Improving direct transfer across benchmark sizes therefore requires greater spatial and temporal diversity in the training data.

\section{Conclusion and Research Implications}\label{sec_conclusion}

This paper develops L-PDDOIs as a structural learning approach to dual stabilization. The framework learns pairwise order relations that restrict the dual space, complementing methods that either predict a single dual vector or derive stabilizing inequalities from problem-specific arguments. CAPSC constructs jointly supported labels from sampled points on the optimal dual face, graph repair reconciles independently predicted relations, and transitive reduction provides a sparse representation of the resulting order cone.

The computational results show that dual-space contraction drives most of the acceleration, while the location of the contracted region remains critical. Although the Dual-Reg and Rand-pair ablations achieve substantial speedups, their considerably larger bound losses indicate that learning improves performance by directing contraction toward regions compatible with the optimal dual face. Graph repair strengthens bound control by limiting unintended relations, whereas transitive reduction preserves the induced order cone while reducing LP and recovery overhead. These findings support two deployment regimes. Direct deployment prioritizes speed at the cost of a potentially weaker lower bound, while recovery restores the baseline CG bound when exact node processing is required. The transfer results further indicate that distribution shifts in spatial structure, temporal constraints, and instance scale can make the predicted order system overly restrictive, thereby increasing the importance of recovery.

Future work should broaden the training distribution across instance sizes, spatial geometries, demand patterns, and time-window structures. Because CAPSC label generation is computationally expensive, promising directions include active instance selection, hard-pair mining, and adaptive sampling of the optimal dual face. L-PDDOIs should also be evaluated within complete BPC trees, where branching and cutting planes may introduce additional dual variables and require a broader inequality design. Finally, the framework could be extended to bounded dual differences and groupwise or hierarchical relations while preserving sparse deployment, tractable consistency control, and verifiable recovery of the baseline CG bound.

\bibliographystyle{informs2014}
\bibliography{reference}

\ECSwitch
\ECHead{Online Appendix}
\setlength{\abovedisplayskip}{5pt plus 2pt minus 2pt}
\setlength{\belowdisplayskip}{5pt plus 2pt minus 2pt}
\setlength{\abovedisplayshortskip}{3pt plus 1pt minus 1pt}
\setlength{\belowdisplayshortskip}{3pt plus 1pt minus 1pt}
\setlength{\textfloatsep}{9pt plus 2pt minus 2pt}
\setlength{\floatsep}{7pt plus 2pt minus 2pt}
\setlength{\intextsep}{7pt plus 2pt minus 2pt}
\phantomsection\label{sec:ec:start}
\renewcommand{\theHsection}{EC.\arabic{section}}
\renewcommand{\theHsubsection}{EC.\arabic{section}.\arabic{subsection}}
\renewcommand{\theHsubsubsection}{EC.\arabic{section}.\arabic{subsection}.\arabic{subsubsection}}
\renewcommand{\theHtable}{EC.\arabic{table}}
\renewcommand{\theHproposition}{EC.\arabic{proposition}}

\section{XML500 Instance Generation}\label{sec:ec:xml500_generation}
We generate XML500 instances with \(n=500\) using the official CVRPLIB XML generator associated with \citet{queiroga202110000} and distributed through CVRPLIB \citep{lima2014cvrplib}. Following the X-series design \citep{uchoa2017new}, the generator varies four categorical factors. These factors are depot placement \(\zeta_{\mathrm{dep}}\) (random, centered, or cornered), customer positions \(\zeta_{\mathrm{pos}}\) (random, clustered, or random clustered), demands \(\zeta_{\mathrm{dem}}\) (unit, bounded, broad-range, quadrant-structured, or few large with many small demands), and average route size \(\zeta_{\mathrm{route}}\) (from very short to ultra-long expected routes). Because the X-series instances do not report their generator settings, we infer a profile \(\boldsymbol\zeta_{\mathrm{gen}}=(\zeta_{\mathrm{dep}},\zeta_{\mathrm{pos}},\zeta_{\mathrm{dem}},\zeta_{\mathrm{route}})\) for each of the 100 instances. Depot placement is identified from the depot coordinates, route size from the customer count and the value of \(k\) encoded in the instance name, and demand family from its support, spatial structure, and distribution. To distinguish potentially overlapping position categories, we calibrate a nearest-centroid geometry classifier using synthetic instances with 100, 500, and 1{,}000 customers and multiple seeds, represented by normalized statistics of nearest-neighbor and five-nearest-neighbor distances. For each inferred profile, we generate seven independent instances with distinct seeds. Five constitute XML500-Learn, yielding 500 instances for pair label generation, feature screening, validation, and model training; the remaining two constitute XML500-Test, yielding 200 instances reserved for solver evaluation. The original X-series instances are retained for transfer evaluation. Complete implementation details, including the inferred profiles, replicate indices, and random seeds, will be provided in the public repository released after acceptance.

\section{Feature Construction and Computational Details}\label{sec:ec:features}

The CVRP pair model is an XGBoost classifier defined over ordered customer pairs. For each pair \((i,j)\), it predicts whether the inequality \(p_i\le p_j\) should enter the candidate set before graph repair and transitive reduction. After feature screening, the deployed representation contains 32 features.

\subsection{Pair Features for CVRP}\label{sec:ec:pair_features}

Throughout this section, \((i,j)\) denotes the ordered customer pair being scored, whereas \(u\) and \(v\) denote generic customers or dummy indices. For a CVRP instance with \(n\) customers, let \((x_u,y_u)\) and \(q_u\) denote the coordinates and demand of customer \(u\), respectively. Let \((x_0,y_0)\) denote the depot coordinates and \(Q>0\) the vehicle capacity. For \(u,v\in\{0,1,\ldots,n\}\), define \(d_{uv}=\sqrt{(x_u-x_v)^2+(y_u-y_v)^2}\). For each \(u\in[n]\), define \(d_u^{\mathrm{dep}}=d_{0u}\) and \(\phi_u=\operatorname{atan2}(y_u-y_0,x_u-x_0)\).

Using \(\epsilon_{\mathrm{num}}=10^{-12}\), define \(\bar q=\max\{n^{-1}\sum_{u=1}^n q_u,\epsilon_{\mathrm{num}}\}\), \(\bar d^{\mathrm{dep}}=\max\{n^{-1}\sum_{u=1}^n d_u^{\mathrm{dep}},\epsilon_{\mathrm{num}}\}\), and \(d_{\max}^{\mathrm{dep}}=\max\{\max_u d_u^{\mathrm{dep}},\epsilon_{\mathrm{num}}\}\). For \(n\ge 2\), let \(\bar d=\max\{2[n(n-1)]^{-1}\sum_{1\le u<v\le n}d_{uv},\epsilon_{\mathrm{num}}\}\), and set \(\bar d=1\) when \(n=1\).

Let \(\mathcal K_i\) contain the \(K_i=\min\{10,n-1\}\) customers nearest to \(i\), with ties broken by customer index. When \(K_i>0\), set \(\mu_i=K_i^{-1}\sum_{u\in\mathcal K_i}d_{iu}\), \(\nu_i^{\min}=\min_{u\in\mathcal K_i}d_{iu}/\bar d\), and \(\nu_i^{\mathrm{mean}}=\mu_i/\bar d\). The remaining descriptors are \(\nu_i^{\mathrm{std}}=\bar d^{-1}\sqrt{K_i^{-1}\sum_{u\in\mathcal K_i}(d_{iu}-\mu_i)^2}\) and \(\nu_i^{\mathrm{load}}=Q^{-1}\sum_{u\in\mathcal K_i}q_u\).

All four descriptors are set to zero when \(K_i=0\). Finally, define the pair saving \(\mathrm{sav}_{ij}=d_i^{\mathrm{dep}}+d_j^{\mathrm{dep}}-d_{ij}\), and let \(\Delta\phi_{ij}\in[-\pi,\pi)\) be the wrapped value of \(\phi_i-\phi_j\). In Table~\ref{tab_ec_full_cvrp_features}, source and destination refer to customers \(i\) and \(j\), respectively.

\begingroup
\scriptsize
\renewcommand{\arraystretch}{0.96}
\setlength{\tabcolsep}{2pt}
\setlength{\LTpre}{2pt}
\setlength{\LTpost}{2pt}
\setlength{\LTcapwidth}{0.94\linewidth}
\begin{longtable}{@{}>{\raggedleft\arraybackslash}p{0.035\linewidth}|
>{\raggedright\arraybackslash}p{0.45\linewidth}
>{\raggedright\arraybackslash}p{0.45\linewidth}@{}}
\caption{Definitions and deployed order of the 32 retained CVRP pair features.}
\label{tab_ec_full_cvrp_features}\\
\toprule
No. & Feature & Definition for ordered pair \((i,j)\) \\
\midrule
\endfirsthead
\toprule
No. & Feature & Definition for ordered pair \((i,j)\) \\
\midrule
\endhead
\midrule
\endfoot
\bottomrule
\endlastfoot
1  & Signed demand difference
   & \(q_i/\bar q-q_j/\bar q\). \\

2  & Difference in depot distance relative to maximum
   & \(d_i^{\mathrm{dep}}/d_{\max}^{\mathrm{dep}}-d_j^{\mathrm{dep}}/d_{\max}^{\mathrm{dep}}\). \\

3  & Absolute depot distance difference
   & \(\lvert d_i^{\mathrm{dep}}/\bar d^{\mathrm{dep}}-d_j^{\mathrm{dep}}/\bar d^{\mathrm{dep}}\rvert\). \\

4  & Signed depot distance difference
   & \(d_i^{\mathrm{dep}}/\bar d^{\mathrm{dep}}-d_j^{\mathrm{dep}}/\bar d^{\mathrm{dep}}\). \\

5  & Difference in demand relative to capacity
   & \(q_i/Q-q_j/Q\). \\

6  & Absolute demand difference
   & \(\lvert q_i/\bar q-q_j/\bar q\rvert\). \\

7  & Difference in normalized mean neighbor distance
   & \(\nu_i^{\mathrm{mean}}-\nu_j^{\mathrm{mean}}\). \\

8  & Absolute difference in depot distance relative to maximum
   & \(\lvert d_i^{\mathrm{dep}}/d_{\max}^{\mathrm{dep}}-d_j^{\mathrm{dep}}/d_{\max}^{\mathrm{dep}}\rvert\). \\

9  & Difference in normalized nearest-neighbor distance
   & \(\nu_i^{\min}-\nu_j^{\min}\). \\

10 & Destination demand
   & \(q_j/\bar q\). \\

11 & Source demand
   & \(q_i/\bar q\). \\

12 & Absolute difference in demand relative to capacity
   & \(\lvert q_i/Q-q_j/Q\rvert\). \\

13 & Source depot distance
   & \(d_i^{\mathrm{dep}}/\bar d^{\mathrm{dep}}\). \\

14 & Destination depot distance
   & \(d_j^{\mathrm{dep}}/\bar d^{\mathrm{dep}}\). \\

15 & Source mean neighbor distance
   & \(\nu_i^{\mathrm{mean}}\). \\

16 & Destination demand relative to capacity
   & \(q_j/Q\). \\

17 & Cosine of the angular difference
   & \(\cos(\Delta\phi_{ij})\). \\

18 & Source depot distance relative to maximum
   & \(d_i^{\mathrm{dep}}/d_{\max}^{\mathrm{dep}}\). \\

19 & Destination nearest-neighbor distance
   & \(\nu_j^{\min}\). \\

20 & Product of demands relative to mean demand
   & \((q_i/\bar q)(q_j/\bar q)\). \\

21 & Destination depot distance relative to maximum
   & \(d_j^{\mathrm{dep}}/d_{\max}^{\mathrm{dep}}\). \\

22 & Mean pairwise distance relative to maximum depot distance
   & \(\bar d/d_{\max}^{\mathrm{dep}}\). \\

23 & Demand dispersion relative to capacity
   & \(\operatorname{sd}_{u\in[n]}(q_u)/Q\). \\

24 & Absolute angular difference
   & \(\lvert\Delta\phi_{ij}\rvert/\pi\). \\

25 & Source demand relative to capacity
   & \(q_i/Q\). \\

26 & Absolute difference in normalized mean neighbor distance
   & \(\lvert\nu_i^{\mathrm{mean}}-\nu_j^{\mathrm{mean}}\rvert\). \\

27 & Scaled vehicle capacity
   & \(Q/1000\). \\

28 & Pair demand relative to capacity
   & \((q_i+q_j)/Q\). \\

29 & Destination neighborhood distance dispersion
   & \(\nu_j^{\mathrm{std}}\). \\

30 & Mean demand relative to capacity
   & \(\bar q/Q\). \\

31 & Product of demands relative to capacity
   & \((q_i/Q)(q_j/Q)\). \\

32 & Source local neighborhood load
   & \(\nu_i^{\mathrm{load}}\). \\
\end{longtable}
\endgroup

\subsection{Learning Model Selection and Deployment Training}
\label{appendix_cvrp_learning_training}

All CVRP learning procedures use a random split by instance of XML500-Learn into 350 training, 75 validation, and 75 reserved test instances, generated with random seed 42. For the learner comparison, each instance contributes a fixed sample of 10,000 ordered pairs, and all four learners use the same representation with 83 candidate features. XGBoost is selected for deployment using the validation split, whereas the reserved test split is used only for final reporting. Table~\ref{tab_ec_cvrp_learner_comparison} reports performance on the reserved test split. Accuracy, precision, recall, and F1 use \(\tau_{\mathrm{pred}}=0.5\), while the area under the receiver operating characteristic curve (AUC) and average precision (AP) aggregate performance across score thresholds.

\begin{table}[htbp]
\centering
\caption{CVRP learner comparison on the reserved test split using the common 83-feature representation.}
\label{tab_ec_cvrp_learner_comparison}
\footnotesize
\setlength{\tabcolsep}{5pt}
\renewcommand{\arraystretch}{1.00}
\begin{tabular}{@{}lrrrrrr@{}}
\toprule
Model & Accuracy & Precision & Recall & F1 & AUC & AP \\
\midrule
Linear classifier & 0.9003 & 0.9405 & 0.8423 & 0.8887 & 0.9717 & 0.9678 \\
Random forest     & 0.9436 & 0.9327 & 0.9492 & 0.9409 & 0.9885 & 0.9872 \\
Neural network    & 0.9490 & \textbf{0.9512} & 0.9403 & 0.9457 & 0.9912 & 0.9903 \\
XGBoost           & \textbf{0.9524} & 0.9462 & \textbf{0.9534} & \textbf{0.9498} & \textbf{0.9921} & \textbf{0.9913} \\
\bottomrule
\end{tabular}

\end{table}

Feature screening retains the 32 deployed features in Table~\ref{tab_ec_full_cvrp_features}. The screened XGBoost model used in all solver experiments attains an AP of $0.9909$ and an F1 score of $0.9484$ on the same reserved test split. The final model is fitted only on the 350 training instances and 3,500,000 rows; validation, reserved test, XML500-Test, and X100 instances are excluded from refitting. The fixed model is used in all XML500-Test and X100 solver experiments.

\subsection{Generation of Synthetic S-GH200 VRPTW Instances}\label{appendix_gh200_generation}

The S-GH200 collection resamples time windows from the official 200-customer Gehring--Homberger instances while preserving customer coordinates, demands, fleet size, vehicle capacity, depot horizon, and service times. We use the R2, C2, and RC2 classes, each comprising ten variants that share the same static data and differ only in their time windows. For class \(g\), let \(H_g\) denote the depot horizon. For template \(v\in\{1,\ldots,10\}\) and customer \(i\), define the ready time \(e_{vi}\), due time \(\ell_{vi}\), width \(w_{vi}=\ell_{vi}-e_{vi}\), and center \(m_{vi}=(e_{vi}+\ell_{vi})/2\). The restrictive template set is \(\mathcal J_i:=\{v:w_{vi}<\operatorname{round}(0.95H_g)\}\). If \(\mathcal J_i\ne\varnothing\), draw \(V_i\sim\operatorname{Unif}(\mathcal J_i)\) and set \(\tilde m_i=m_{V_i i}\); otherwise, set \(\tilde m_i=H_g/2\). Sampling template indices preserves the multiplicities of repeated centers without adding center jitter. For each \(v\in\mathcal J_i\), define \(\hat e_{vi}:=\max\{0,\min\{\operatorname{round}(\tilde m_i-w_{vi}/2),H_g-w_{vi}\}\}\), and generate
\[
(\tilde e_{vi},\tilde\ell_{vi})=
\begin{cases}
(\hat e_{vi},\hat e_{vi}+w_{vi}), & v\in\mathcal J_i,\\
(0,H_g), & v\notin\mathcal J_i.
\end{cases}.
\]
Restrictive windows retain their original widths, while nonrestrictive windows span the full depot horizon; service times are inherited from the class template. We generate ten independent batches per class for S-GH200-Learn and two per class for S-GH200-Test, with each batch containing one instance for each of the ten time window templates. The procedure yields 300 learning instances and 60 test instances. A manifest records the source template and random seed, while the original Gehring--Homberger instances are reserved for transfer evaluation.

\subsection{VRPTW Features and Learning Protocol}
\label{sec:ec:vrptw_features}

The VRPTW model extends the 32 CVRP features in Table~\ref{tab_ec_full_cvrp_features} with 24 time window descriptors and three spatial regime indicators, yielding 59 deployed features. We fit a full model using all candidate time window descriptors and retain the 24 with the largest mean split gains. For customer \(u\), let \(e_u\), \(\ell_u\), and \(\tau_u\) denote ready time, due time, and service time, respectively; subscript 0 denotes the depot, and Euclidean distance \(d_{uv}\) is used as travel time. Define \(T_{\min}=\min\{e_0,\min_u e_u\}\), \(T_{\max}=\max\{\ell_0,\max_u\ell_u\}\), and \(H=\max\{T_{\max}-T_{\min},\epsilon_{\mathrm{num}}\}\). For each customer, let \(w_u=\max\{0,\ell_u-e_u\}\), \(L_u=\max\{e_u,e_0+d_u^{\mathrm{dep}}\}\), \(R_u=\ell_0-\tau_u-d_u^{\mathrm{dep}}\), \(\sigma_u=[\min\{\ell_u,R_u\}-L_u]/H\), and \(\sigma_u^+=\max\{0,\sigma_u\}\).

For distinct customers \(u\) and \(v\), let \(s_{uv}=\min\{\ell_u,\ell_v-\tau_u-d_{uv},\ell_0-\tau_u-d_{uv}-\tau_v-d_v^{\mathrm{dep}}\}-L_u\). Superscripts \(T\) and \(TC\) denote time compatibility and joint time and capacity compatibility, respectively. With \(\varepsilon_t=10^{-9}\) and \(n_{\mathrm{oth}}=\max\{n-1,1\}\), define \(F^T_{uv}=\mathbf 1\{s_{uv}\ge-\varepsilon_t,\ e_v\le R_v+\varepsilon_t\}\), \(F^{TC}_{uv}=F^T_{uv}\mathbf 1\{q_u+q_v\le Q\}\), \(D_u^{TC,+}=n_{\mathrm{oth}}^{-1}\sum_{v\ne u}F^{TC}_{uv}\), and \(D_u^{TC,-}=n_{\mathrm{oth}}^{-1}\sum_{v\ne u}F^{TC}_{vu}\).

Let \(\mathcal F_u^{TC}=\{v\ne u:F^{TC}_{uv}=1\}\). The pair route cost is
\[
c_u^{\mathrm{pair}}=
\begin{cases}
\displaystyle \min_{v\in\mathcal F_u^{TC}}\frac{d_u^{\mathrm{dep}}+d_{uv}+d_v^{\mathrm{dep}}}{\bar d}, & \mathcal F_u^{TC}\ne\varnothing,\\[2pt]
\displaystyle \frac{2d_u^{\mathrm{dep}}}{\bar d}, & \mathcal F_u^{TC}=\varnothing,
\end{cases}.
\]
The maximum saving proxy is \(\mathrm{sav}_u^{\max}=\max\{0,\max_{v\in\mathcal F_u^{TC}}\mathrm{sav}_{uv}/\bar d\}\), where the inner maximum is zero when \(\mathcal F_u^{TC}=\varnothing\). With \(o_{uv}=\max\{0,\min(\ell_u,\ell_v)-\max(e_u,e_v)\}\), define
\[
\begin{aligned}
P_u^{\mathrm{ov}}&=\frac{n_{\mathrm{oth}}^{-1}\sum_{v\ne u}\mathbf 1\{o_{uv}>0\}q_v/Q}{\max\{\sigma_u^+,10^{-6}\}},\\
\Psi_u&=\frac{q_u}{Q}+\frac{\tau_u}{\max\{w_u,\epsilon_{\mathrm{num}}\}}+c_u^{\mathrm{pair}}+P_u^{\mathrm{ov}}-\sigma_u^+
-\frac{D_u^{TC,+}+D_u^{TC,-}+\mathrm{sav}_u^{\max}}{2}.
\end{aligned}
\]

For the remaining selected descriptors, let \(\mathfrak C_u=\{\!\{d_u^{\mathrm{dep}}+d_{uv}+d_v^{\mathrm{dep}}:v\in\mathcal F_u^{TC}\}\!\}\), and denote its sorted elements by \(c_{u,(1)}\le\cdots\le c_{u,(|\mathfrak C_u|)}\). Then
\[
 A_u^{(k)}=
 \begin{cases}
 \displaystyle\frac{1}{\bar d\min\{k,|\mathfrak C_u|\}}
 \sum_{a=1}^{\min\{k,|\mathfrak C_u|\}}c_{u,(a)},&\mathfrak C_u\ne\varnothing,\\[3pt]
 \displaystyle\frac{2d_u^{\mathrm{dep}}}{\bar d},&\mathfrak C_u=\varnothing.
 \end{cases}
\]
Define the solo route cost \(c_u^{\mathrm{solo}}=2d_u^{\mathrm{dep}}/\bar d\), normalized overlap \(O_u=n_{\mathrm{oth}}^{-1}\sum_{v\ne u}o_{uv}/H\), and mutual time compatibility degree \(D_u^{T,\leftrightarrow}=n_{\mathrm{oth}}^{-1}\sum_{v\ne u}F^T_{uv}F^T_{vu}\). The residual capacity proxy is
\[
 \chi_u=
 \begin{cases}
 \displaystyle\min_{v:F^T_{uv}=1}\frac{\max\{0,Q-q_u-q_v\}}{Q},
 &\sum_{v\ne u}F^T_{uv}>0,\\[3pt]
 0,&\text{otherwise}.
 \end{cases}
\]

For the spatial grid, let \(x_{\min}=\min_u x_u\), \(x_{\max}=\max_u x_u\), and \(\Delta_x=\max\{x_{\max}-x_{\min},\epsilon_{\mathrm{num}}\}\), with \(y_{\min}\), \(y_{\max}\), and \(\Delta_y\) defined analogously. The clamped indices are \(g_x(u)=\min\{9,\max\{0,\lfloor 10(x_u-x_{\min})/\Delta_x\rfloor\}\}\), with \(g_y(u)\) defined analogously. If \(N_{\mathrm{occ}}\) denotes the number of occupied cells \((g_x(u),g_y(u))\), then \(\rho_{10}=N_{\mathrm{occ}}/100\). Table~\ref{tab_ec_full_vrptw_features} lists the 27 additional features in deployed order, where source and destination refer to customers \(i\) and \(j\), respectively.

\begingroup
\scriptsize
\renewcommand{\arraystretch}{0.96}
\setlength{\tabcolsep}{2pt}
\setlength{\LTpre}{2pt}
\setlength{\LTpost}{2pt}
\setlength{\LTcapwidth}{0.94\linewidth}
\begin{longtable}{@{}>{\raggedleft\arraybackslash}p{0.035\linewidth}|>{\raggedright\arraybackslash}p{0.45\linewidth}>{\raggedright\arraybackslash}p{0.45\linewidth}@{}}
\caption{Definitions and deployed order of the 27 additional VRPTW features.}
\label{tab_ec_full_vrptw_features}\\
\toprule
No. & Feature & Definition for ordered pair \((i,j)\) \\
\midrule
\endfirsthead
\toprule
No. & Feature & Definition for ordered pair \((i,j)\) \\
\midrule
\endhead
\midrule
\endfoot
\bottomrule
\endlastfoot
\multicolumn{3}{@{}l}{\textit{Time window features}}\\
\midrule
33 & Signed dual hardness difference & \(\Psi_i-\Psi_j\). \\
34 & Source service time relative to window width & \(\tau_i/\max\{w_i,\epsilon_{\mathrm{num}}\}\). \\
35 & Destination window width relative to horizon & \(w_j/H\). \\
36 & Destination best pair saving & \(\mathrm{sav}_j^{\max}\). \\
37 & Source best pair route cost & \(c_i^{\mathrm{pair}}\). \\
38 & Destination best pair route cost & \(c_j^{\mathrm{pair}}\). \\
39 & Source dual hardness & \(\Psi_i\). \\
40 & Source solo-route cost & \(c_i^{\mathrm{solo}}\). \\
41 & Destination time-window center & \(((e_j+\ell_j)/2-T_{\min})/H\). \\
42 & Destination mean top-five pair-route cost & \(A_j^{(5)}\). \\
43 & Signed time-window overlap sum & \(O_i-O_j\). \\
44 & Destination minimum compatible capacity residual & \(\chi_j\). \\
45 & Destination solo-route cost & \(c_j^{\mathrm{solo}}\). \\
46 & Signed mutual time-compatibility degree & \(D_i^{T,\leftrightarrow}-D_j^{T,\leftrightarrow}\). \\
47 & Destination ready time relative to horizon & \((e_j-T_{\min})/H\). \\
48 & Destination service time relative to window width & \(\tau_j/\max\{w_j,\epsilon_{\mathrm{num}}\}\). \\
49 & Destination due time relative to horizon & \((\ell_j-T_{\min})/H\). \\
50 & Signed window-width difference & \((w_i-w_j)/H\). \\
51 & Destination mean top-ten pair-route cost & \(A_j^{(10)}\). \\
52 & Source best pair saving & \(\mathrm{sav}_i^{\max}\). \\
53 & Source time-window center & \(((e_i+\ell_i)/2-T_{\min})/H\). \\
54 & Destination standalone slack & \(\sigma_j\). \\
55 & Source ready time relative to horizon & \((e_i-T_{\min})/H\). \\
56 & Source mean top-ten pair-route cost & \(A_i^{(10)}\). \\
\midrule
\multicolumn{3}{@{}l}{\textit{Spatial regime indicators}}\\
\midrule
57 & C2 spatial regime & \(\mathbf 1\{\rho_{10}\le 0.75\}\). \\
58 & R2 spatial regime & \(\mathbf 1\{\rho_{10}\ge 0.825\}\). \\
59 & RC2 spatial regime & \(\mathbf 1\{0.75<\rho_{10}<0.825\}\). \\
\end{longtable}
\endgroup

The protocol splits instances within each class. S-GH200-Learn instances 001--080 form the 240-instance training set, while instances 081--100 form the 60-instance validation set. S-GH200-Test instances 001--020 form the 60-instance test set. Each training instance contributes 10,000 ordered pairs, yielding 2,400,000 training rows. S-GH200-Test is excluded from fitting and validation and is used for prediction reporting and solver evaluation. All four candidate learners use the same split, sampled pairs, and 59 features in Tables~\ref{tab_ec_full_cvrp_features} and~\ref{tab_ec_full_vrptw_features}. Model selection maximizes validation AUC and uses AP to break ties, which selects XGBoost for deployment. Table~\ref{tab_vrptw_prediction_quality} reports performance on S-GH200-Test. Accuracy, precision, recall, and F1 use \(\tau_{\mathrm{pred}}=0.5\), while AUC and AP aggregate performance across score thresholds.

\begin{table}[htbp]
\centering
\caption{VRPTW learner comparison with 59 features.}
\label{tab_vrptw_prediction_quality}
\footnotesize
\setlength{\tabcolsep}{5pt}
\renewcommand{\arraystretch}{1.00}
\begin{tabular}{@{}lrrrrrr@{}}
\toprule
Model & Accuracy & Precision & Recall & F1 & AUC & AP \\
\midrule
Logistic regression & 0.7203 & 0.5406 & 0.8320 & 0.6554 & 0.8428 & 0.7380 \\
Random forest       & 0.8525 & 0.7297 & 0.8553 & 0.7876 & 0.9300 & 0.8687 \\
Neural network      & 0.8266 & 0.8069 & 0.6015 & 0.6892 & 0.8908 & 0.8131 \\
XGBoost             & \textbf{0.9365} & \textbf{0.8803} & \textbf{0.9275} & \textbf{0.9033} & \textbf{0.9841} & \textbf{0.9709} \\
\bottomrule
\end{tabular}
\end{table}

\section{Proofs of Statements in Section~\ref{sec_pair_layer}}\label{sec:ec:pair_proofs}

\subsection{Proof of Proposition~\ref{prop:capsc_properties}}\label{sec:ec:proof_capsc_properties}

\textit{Proof.} Let \((\kappa^\star,\eta^\star)\) be an optimal CAPSC solution, let \(\mathcal T^\star:=\{k\in[K]:\eta_k^\star=1\}\), and let \(E^{\mathrm{CAPSC}}:=\{(i,j)\in\mathcal I:\kappa_{ij}^\star=1\}\). Since \(\lceil\alpha K\rceil\ge1\), the retention constraint implies \(\mathcal T^\star\ne\varnothing\).

Consider any selected arc \((i,j)\in E^{\mathrm{CAPSC}}\). Since \(\kappa_{ij}^\star=1\), Constraint~\eqref{eq:CAPSC_support} gives \(\sum_{k\in\mathcal V_{ij}}\eta_k^\star\le 0\). No retained sample can therefore belong to \(\mathcal V_{ij}\), and every \(k\in\mathcal T^\star\) satisfies \(p_i^{(k)}+\varepsilon\le p_j^{(k)}\).

Since \(\mathcal T^\star\ne\varnothing\), a retained sample lies in \(\mathcal D^\star\) and satisfies every inequality in \(E^{\mathrm{CAPSC}}\). Hence \(E^{\mathrm{CAPSC}}\) is jointly deep dual-optimal.

We first prove acyclicity. Suppose, for contradiction, that \(G^{\mathrm{CAPSC}}\) contains a directed cycle \(v_0\to v_1\to\cdots\to v_{L-1}\to v_0\), where \(L\ge2\). Fix any \(k\in\mathcal T^\star\). Applying the margin support inequality to every arc in the cycle and summing gives
\[
\sum_{a=0}^{L-1}p_{v_a}^{(k)}+L\varepsilon
\le
\sum_{a=0}^{L-1}p_{v_a}^{(k)}.
\]
The resulting inequality \(L\varepsilon\le0\) contradicts \(L\ge2\) and \(\varepsilon>0\), which proves that \(G^{\mathrm{CAPSC}}\) is acyclic.

We next prove inference closure. Suppose that \(G^{\mathrm{CAPSC}}\) contains a directed path \(i=v_0\to v_1\to\cdots\to v_L=j\), where \(L\ge1\). The claim is immediate for \(L=1\). For \(L\ge2\), summing the margin inequalities along the path gives \(p_i^{(k)}+L\varepsilon\le p_j^{(k)}\) for every \(k\in\mathcal T^\star\), and thus \(p_i^{(k)}+\varepsilon\le p_j^{(k)}\). Hence no retained sample belongs to \(\mathcal V_{ij}\), so setting \(\kappa_{ij}=1\) preserves Constraint~\eqref{eq:CAPSC_support}.

We also have \(\kappa_{ji}^\star=0\); otherwise, the arc \(j\to i\) and the path from \(i\) to \(j\) would form a directed cycle. Adding \((i,j)\) therefore preserves Constraint~\eqref{eq:CAPSC_antisymmetry}. If \(\kappa_{ij}^\star=0\), setting it to \(1\) preserves feasibility and increases the objective by one, contradicting optimality. Thus \(\kappa_{ij}^\star=1\), and the arbitrary path proves inference closure. \hfill$\square$

\subsection{Proof of Proposition~\ref{prop_pf_feasibility}}\label{sec:ec:proof_pf_feasibility}

\textit{Proof.} Let \(E^{\mathrm{rep}}:=\{(i,j)\in E^0:\delta_{ij}=1\}\), and let \(G^{\mathrm{rep}}=G(E^{\mathrm{rep}})\) be the directed graph induced by a feasible repair solution \(\delta\).

We first prove inference closure. Suppose \((i,j)\in E^{\mathrm{rep}}\) and \((j,k)\in E^{\mathrm{rep}}\). Then \(\delta_{ij}=\delta_{jk}=1\), \(j\in N_{E^0}^+(i)\), and \(k\in N_{E^0}^+(j)\). If \(k\notin N_{E^0}^+(i)\), Constraint~\eqref{eq_postprocessing_reduced_cons2} gives \(\delta_{ij}+\delta_{jk}\le 1\), a contradiction. Hence \(k\in N_{E^0}^+(i)\cap N_{E^0}^+(j)\), and Constraint~\eqref{eq_postprocessing_reduced_cons1} gives \(\delta_{ij}+\delta_{jk}-\delta_{ik}\le 1\). Substituting \(\delta_{ij}=\delta_{jk}=1\) yields \(\delta_{ik}=1\). Thus every retained path of length two contains its shortcut.

Consider a shortest directed path between two distinct reachable nodes. If the path has more than one arc, the length two argument applied to its first two arcs produces a shorter path with the same endpoints. Hence every shortest path has one arc, and every reachable pair belongs to \(E^{\mathrm{rep}}\). Thus, \(G^{\mathrm{rep}}\) has inference closure.

We next prove acyclicity. Suppose, for contradiction, that \(G^{\mathrm{rep}}\) contains a directed cycle \(i_1 \to i_2 \to \cdots \to i_\ell \to i_1\), where \(\ell\ge 2\). Inference closure implies \((i_1,i_\ell)\in E^{\mathrm{rep}}\), while the cycle contains \((i_\ell,i_1)\in E^{\mathrm{rep}}\). Thus \(i_\ell\in N_{E^0}^+(i_1)\) and \(i_1\in N_{E^0}^+(i_\ell)\). Because \(E^0\subseteq\mathcal I\), we also have \(i_1\notin N_{E^0}^+(i_1)\). Applying Constraint~\eqref{eq_postprocessing_reduced_cons2} to \((i_1,i_\ell,i_1)\) gives \(\delta_{i_1 i_\ell}+\delta_{i_\ell i_1}\le 1\), which forbids selecting both arcs. The contradiction proves that \(G^{\mathrm{rep}}\) is acyclic.

Every feasible repair solution therefore induces a DAG with inference closure. \hfill$\square$

\subsection{SCC Preprocessing and Feasible Extension}\label{sec:ec:sccPreprocessing}

Algorithm~\ref{alg:ecSccRepair} gives the preprocessing procedure used before exact repair. Let \(\mathcal B=(B_1,\ldots,B_{N_{\mathrm{blk}}})\) be an ordered partition of \([m]\), obtained by topologically sorting the SCC condensation of \(G^0\). An SCC larger than \(K_{\mathrm{scc}}\) is split into ordered chunks of size at most \(K_{\mathrm{scc}}\). Within such an SCC \(C\), nodes are ranked by \(\omega(v):=d_C^+(v)-d_C^-(v)\), with ties broken by larger \(d_C^+(v)\), smaller \(d_C^-(v)\), and then customer index. The first \(K_{\mathrm{scc}}\) nodes form the next chunk, the scores are recomputed on the remaining nodes, and the procedure continues until \(C\) is exhausted. For chunks from the same SCC, arcs from a later to an earlier chunk are removed. Write \(v\prec_{\mathcal B}u\) when the block containing \(v\) precedes the block containing \(u\).

We use \(K_{\mathrm{scc}}=50\). Heuristic arcs are fixed to one; all other retained within-block and forward cross-block arcs enter the residual PF, except variables fixed to zero when a fixed-one arc would imply a closure arc absent from \(E^0\). PF rows are generated explicitly before optimization. Gurobi uses one thread, zero relative mixed-integer programming gap, and no time limit, and must return an optimal solution.

{
\DontPrintSemicolon
\begin{algorithm}[H]
\small
\caption{SCC preprocessing for feasible partial repair}\label{alg:ecSccRepair}
\KwData{Raw graph \(G^0=([m],E^0)\), neighborhoods \(N_{E^0}^+(i)\), and ordered blocks \(\mathcal B\).}
\KwResult{A partial repaired arc set \(E^{\mathrm H}\).}
Set \(E^{\mathrm H}\leftarrow\varnothing\) and \(\operatorname{Desc}(i)\leftarrow\{i\}\) for every \(i\in[m]\)\;
Choose any node order \((i_1,\ldots,i_m)\) consistent with \(\mathcal B\)\;
\For{\(k=m,m-1,\ldots,1\)}{
    Set \(u\leftarrow i_k\)\;
    \ForEach{\(v\prec_{\mathcal B}u\) satisfying \(\operatorname{Desc}(u)\subseteq N_{E^0}^+(v)\)}{
        Set \(E^{\mathrm H}\leftarrow E^{\mathrm H}\cup\{(v,u)\}\)\;
        Set \(\operatorname{Desc}(v)\leftarrow\operatorname{Desc}(v)\cup\operatorname{Desc}(u)\)\;
    }
}
\Return \(E^{\mathrm H}\)\;
\end{algorithm}
}

\begin{proposition}\label{prop:ecSccRepair}
Algorithm~\ref{alg:ecSccRepair} returns an arc set \(E^{\mathrm H}\subseteq E^0\) for which \(G(E^{\mathrm H})\) is a DAG with inference closure. The incidence vector of \(E^{\mathrm H}\) is feasible for \eqref{eq_postprocessing_reduced}. Consequently, fixing \(\delta_{ij}=1\) for every \((i,j)\in E^{\mathrm H}\) leaves the original PF formulation feasible.
\end{proposition}

\textit{Proof.} We first establish an invariant of the reverse scan. Immediately before a node \(u\) is processed, \(\operatorname{Desc}(u)\) consists of \(u\) and all nodes reachable from \(u\) through arcs already selected by the algorithm. After \(u\) has been processed, \(\operatorname{Desc}(u)\) does not change.

The invariant follows by reverse induction over the node order consistent with \(\mathcal B\). Every selected arc leaving \(u\) points to a later block and is considered when its head is processed, before \(u\) itself. Whenever \((u,w)\) is selected, the update \(\operatorname{Desc}(u)\leftarrow\operatorname{Desc}(u)\cup\operatorname{Desc}(w)\) adds exactly the nodes reachable through \(w\). All possible heads in later blocks have already been scanned when \(u\) is processed, so no selected arc can subsequently change \(\operatorname{Desc}(u)\).

Whenever the algorithm selects \((v,u)\), the inclusion \(u\in\operatorname{Desc}(u)\) and the acceptance condition \(\operatorname{Desc}(u)\subseteq N_{E^0}^+(v)\) imply \((v,u)\in E^0\). Hence \(E^{\mathrm H}\subseteq E^0\). Every selected arc also moves from an earlier block to a later block, so \(G(E^{\mathrm H})\) is acyclic.

It remains to prove inference closure. Suppose that \((v,u),(u,w)\in E^{\mathrm H}\). The block order gives \(v\prec_{\mathcal B}u\prec_{\mathcal B}w\), so \(w\) is processed before \(u\). Selection of \((u,w)\) makes \(\operatorname{Desc}(u)\) contain \(\operatorname{Desc}(w)\). When \((v,u)\) is subsequently selected, the acceptance condition therefore gives \(\operatorname{Desc}(w)\subseteq\operatorname{Desc}(u)\subseteq N_{E^0}^+(v)\). The invariant implies that \(\operatorname{Desc}(w)\) already has its final value when \(w\) is processed. Since \(v\prec_{\mathcal B}w\), the inner loop considers \(v\) while processing \(w\), and the preceding inclusion causes the algorithm to select \((v,w)\). Thus every selected path of length two contains its shortcut. Induction on path length proves inference closure.

Let \(\delta^{\mathrm H}\) denote the incidence vector of \(E^{\mathrm H}\). Constraint~\eqref{eq_postprocessing_reduced_cons1} is satisfied because every selected path of length two is accompanied by its shortcut. A violation of Constraint~\eqref{eq_postprocessing_reduced_cons2} would require selected arcs whose shortcut is absent from \(E^0\); however, inference closure ensures that this shortcut belongs to \(E^{\mathrm H}\subseteq E^0\). Therefore, \(\delta^{\mathrm H}\) is feasible for PF. After fixing its selected variables to one, the same incidence vector remains feasible for the original formulation. \hfill$\square$

\subsection{Proof of Theorem~\ref{theorem_same_feasible_region}}\label{sec:ec:proof_same_feasible_region}

\textit{Proof.} Assume first that \(G(E_1)\) and \(G(E_2)\) have the same reachability relation. We show \(\mathcal O(E_2)\subseteq \mathcal O(E_1)\); the reverse inclusion is symmetric. Let \(\mathbf p\in \mathcal O(E_2)\) and \((i,j)\in E_1\). The arc \((i,j)\) is a path in \(G(E_1)\), so \(G(E_2)\) contains a path \(i=v_0 \to v_1 \to \cdots \to v_t=j\). Since \(\mathbf p\in \mathcal O(E_2)\), we have \(p_{v_0}\le p_{v_1}\le \cdots \le p_{v_t}\), and hence \(p_i\le p_j\). The choice of \((i,j)\) was arbitrary, which proves \(\mathcal O(E_2)\subseteq \mathcal O(E_1)\). The reverse inclusion follows symmetrically, giving \(\mathcal O(E_1)=\mathcal O(E_2)\).

For the converse, suppose that the reachability relations differ. Without loss of generality, let distinct nodes \(i,j\in[m]\) satisfy \(j\in R_{E_1}^+(i)\) and \(j\notin R_{E_2}^+(i)\). Adding \((j,i)\) to the DAG \(G(E_2)\) cannot create a cycle, since such a cycle would contain an existing path from \(i\) to \(j\). Let \(\prec\) be a topological order of the augmented DAG \(G(E_2)\cup\{(j,i)\}\), and define \(p_v:=\operatorname{pos}_{\prec}(v)\) for every \(v\in[m]\), where \(\operatorname{pos}_{\prec}(v)\) is the position of \(v\) in \(\prec\).

Every arc \((u,v)\in E_2\) respects \(\prec\), so \(p_u<p_v\) and \(\mathbf p\in \mathcal O(E_2)\). The added arc \((j,i)\) also respects \(\prec\), which gives \(p_j<p_i\). Thus \(\mathbf p\) violates \(p_i\le p_j\). Since \(j\in R_{E_1}^+(i)\), every point in \(\mathcal O(E_1)\) satisfies this inequality. Consequently, \(\mathbf p\notin \mathcal O(E_1)\), and \(\mathcal O(E_1)\ne \mathcal O(E_2)\), which proves the converse. \hfill$\square$

\subsection{Proof of Proposition~\ref{prop_bound_recovery_certificate}}\label{sec:ec:proof_bound_recovery_certificate}

\textit{Proof.} For each pair \(e=(i,j)\in\mathcal E\), let \(\mathbf b_e=\mathbf e_i-\mathbf e_j\). After CG convergence, the master augmented with pair variables for \(\mathcal E\) can be written as
\[
z(\mathcal E)=\min_{x,\xi}\left\{\sum_{r\in\Omega}c_rx_r:
\sum_{r\in\Omega}a_{ir}x_r+\sum_{e\in\mathcal E}b_{ie}\xi_e=1\ \forall i\in[m],\
x_r\ge0\ \forall r\in\Omega,\ \xi_e\ge0\ \forall e\in\mathcal E\right\}.
\]
Its dual is
\[
z(\mathcal E)=\max_{\mathbf p}\left\{\sum_{i\in[m]}p_i:
\sum_{i\in[m]}a_{ir}p_i\le c_r\ \forall r\in\Omega,\
p_i-p_j\le0\ \forall(i,j)\in\mathcal E,\ p_i\in\mathbb R\ \forall i\in[m]\right\}.
\]
By the finite optimality assumptions, strong LP duality applies, and the two displayed programs have the common value \(z(\mathcal E)\). If \(\mathcal E'\subseteq\mathcal E\), the dual feasible region for \(\mathcal E'\) contains that for \(\mathcal E\). Maximization over the larger region gives \(z(\mathcal E')\ge z(\mathcal E)\).

We next prove the equivalence. Suppose first that \(\mathcal E\) is jointly deep dual-optimal. Some \(\bar{\mathbf p}\in\mathcal D^\star\) then satisfies every inequality indexed by \(\mathcal E\), so it is feasible for the augmented dual and attains \(z^\star\). The augmented dual feasible region is contained in the original dual feasible region, which gives \(z(\mathcal E)\le z^\star\). Feasibility of \(\bar{\mathbf p}\) gives the reverse inequality, and therefore \(z(\mathcal E)=z^\star\).

Conversely, suppose that \(z(\mathcal E)=z^\star\), and let \(\mathbf p^{\mathcal E}\) be any optimal solution of the augmented dual. The vector \(\mathbf p^{\mathcal E}\) satisfies all original route constraints and has objective value \(z^\star\), so \(\mathbf p^{\mathcal E}\in\mathcal D^\star\). It also satisfies every pair inequality indexed by \(\mathcal E\). Hence \(\mathcal E\) is jointly deep dual-optimal, and every augmented dual optimum belongs to \(\mathcal D^\star\).

Still assuming \(z(\mathcal E)=z^\star\), let \(x^\star\) be an optimal solution of the original master. The solution \((x^\star,\mathbf 0)\) is feasible for the augmented master and has objective value \(z^\star=z(\mathcal E)\). It is therefore an augmented master optimum with \(\xi_e=0\) for every \(e\in\mathcal E\).

Finally, suppose that an optimal augmented master solution \((\bar x,\bar\xi)\) satisfies \(\bar\xi_e=0\) for every \(e\in\mathcal E\). The vector \(\bar x\) is feasible for the original master and has objective value \(z(\mathcal E)\), which gives \(z^\star\le z(\mathcal E)\). The augmented master contains every original master solution by setting \(\xi=0\), so \(z(\mathcal E)\le z^\star\). Hence \(z(\mathcal E)=z^\star\), and \(\mathcal E\) is a certified L-PDDOI set. \hfill$\square$

\subsection{Finite Termination of Exact Recovery}\label{sec:ec:proofRecoveryTermination}

\begin{corollary}\label{cor:recoveryTermination}
Let \(\mathcal E^{(0)}\) be finite, \(J\in\mathbb Z_{\ge1}\), and \(K_{\mathrm{tail}}\in\mathbb Z_{\ge0}\), and suppose that the assumptions of Proposition~\ref{prop_bound_recovery_certificate} hold for \(\mathcal E^{(0)}\). Assume that every pricing stage terminates, the final stage solves the current augmented master to full CG optimality, and Algorithm~\ref{alg:pair_recovery} uses exact LP solutions, exact activity detection, and no gap target stopping. The algorithm terminates after at most \(|\mathcal E^{(0)}|\) release rounds and returns a set \(\mathcal E^{(t)}\) satisfying \(z(\mathcal E^{(t)})=z^\star\).
\end{corollary}

\textit{Proof.} Every subset of \(\mathcal E^{(0)}\) yields a feasible augmented master because an original master solution remains feasible with all pair variables equal to zero. Its feasible region is contained in that of the augmentation by \(\mathcal E^{(0)}\), so its objective is also bounded below. Thus the assumptions of Proposition~\ref{prop_bound_recovery_certificate} hold at every recovery round.

Whenever \(\mathcal E_{\mathrm{act}}^{(t)}\ne\varnothing\), the released set \(\mathcal E_{\mathrm{rel}}^{(t)}\) is nonempty, regardless of which branch of the tailing off rule applies. Hence \(|\mathcal E^{(t+1)}|<|\mathcal E^{(t)}|\), and at most \(|\mathcal E^{(0)}|\) release rounds can occur.

If the algorithm terminates at the final pricing stage with \(\mathcal E_{\mathrm{act}}^{(t)}=\varnothing\), full CG optimality and exact activity detection give an optimal augmented master solution with \(\xi_e=0\) for every \(e\in\mathcal E^{(t)}\). Proposition~\ref{prop_bound_recovery_certificate} then yields \(z(\mathcal E^{(t)})=z^\star\). Otherwise, repeated releases eventually produce \(\mathcal E^{(t)}=\varnothing\). The augmented master then coincides with the original master, and the final stage returns \(z(\varnothing)=z^\star\). Thus the algorithm terminates with the baseline CG bound after at most \(|\mathcal E^{(0)}|\) releases. \hfill$\square$

\end{document}